\documentclass[10pt]{amsart}
\usepackage{newlfont,amsfonts,amssymb,amsmath,amsthm,amsgen,amscd,enumerate,datetime}

\usepackage{euscript,color}

\newcommand{\der}{{ d}}

\newcommand{\QQ}{\ensuremath{\mathbb{Q}}}
\newcommand{\N}{\ensuremath{\mathbb{N}}}
\newcommand{\Z}{\ensuremath{\mathbb{Z}}}
\newcommand{\C}{\ensuremath{\mathbb{C}}}

\renewcommand{\O}{\ensuremath{\mathrm{O}}}
\newcommand{\R}{\ensuremath{\mathbb{R}}}
\newcommand{\Nor}{\ensuremath{\mathrm{Nor}}}
\newcommand{\Stab}{\ensuremath{\mathrm{Stab}}}
\newcommand{\pr}{\ensuremath{\mathrm{pr}}}
\newcommand{\U}{\ensuremath{\mathrm{U}}}
\newcommand{\G}{\ensuremath{\mathrm{G}}}
\newcommand{\Spin}{\ensuremath{\mathrm{Spin}}}
\newcommand{\SO}{\mathrm{SO}}
\newcommand{\SU}{\mathrm{SU}}
\newcommand{\Sp}{\mathrm{Sp}}
\newcommand{\1}{\mathbf{1}}

\newcommand{\e}{\mathrm{e}}
\renewcommand{\d}{{d}}
\newcommand{\bcase}{\begin{case}}
\newcommand{\ecase}{\end{case}}

\newcommand{\bclaim}{\begin{claim}}
\newcommand{\eclaim}{\end{claim}}

\newcommand{\bstep}{\begin{step}}
\newcommand{\estep}{\end{step}}

\newcommand{\bhlem}{\begin{hlem}}
\newcommand{\ehlem}{\end{hlem}}

\newcommand{\bleer}{\begin{leer}}
\newcommand{\eleer}{\end{leer}}
\newcommand{\bde}{\begin{de}}
\newcommand{\ede}{\end{de}}
\newcommand{\ol}{\overline}

\newcommand{\mf}{\mathfrak}
\newcommand{\bs}{\begin{satz}}
\newcommand{\es}{\end{satz}}
\newcommand{\btheo}{\begin{theo}}
\newcommand{\etheo}{\end{theo}}
\newcommand{\bfolg}{\begin{folg}}
\newcommand{\efolg}{\end{folg}}
\newcommand{\blem}{\begin{lem}}
\newcommand{\elem}{\end{lem}}
\newcommand{\bnote}{\begin{note}}
\newcommand{\enote}{\end{note}}
\newcommand{\bprf}{\begin{proof}}
\newcommand{\eprf}{\end{proof}}
\newcommand{\bd}{\begin{displaymath}}
\newcommand{\ed}{\end{displaymath}}
\newcommand{\be}{\begin{eqnarray*}}
\newcommand{\ee}{\end{eqnarray*}}
\newcommand{\eeqa}{\end{eqnarray}}
\newcommand{\beqa}{\begin{eqnarray}}
\newcommand{\bi}{\begin{itemize}}
\newcommand{\ei}{\end{itemize}}
\newcommand{\bnum}{\begin{enumerate}}
\newcommand{\enum}{\end{enumerate}}
\newcommand{\la}{\langle}
\newcommand{\ra}{\rangle}

\newcommand{\sums}{\sum\limits}

\newcommand{\beq}{\begin{equation}}
\newcommand{\eeq}{\end{equation}}
\newcommand{\einhalb}{\frac{1}{2}}
\newcommand{\rr}{\mathbb{R}}

\newcommand{\vf}{\varphi}
\newcommand{\earr}{\end{array}\]}
\newcommand{\barr}{\[\begin{array}}
\newcommand{\bvec}{\left(\begin{array}{c}}
\newcommand{\evec}{\end{array}\right)}

\newcommand{\lag}{\mathfrak{g}}
\newcommand{\g}{\mathfrak{g}}

\newcommand{\h}{\mathfrak{h}}
\newcommand{\lah}{\mathfrak{h}}

\newcommand{\hol}{\mathfrak{hol}}
\newcommand{\Hol}{\mathrm{Hol}}
\newcommand{\+}{\oplus}

\newcommand{\rrn}{\mathbb{R}^n}
\newcommand{\so}{\mathfrak{so}}

\newcommand{\s}{\sigma}
\newcommand{\del}{\partial}

\newcommand{\bbR}{\mathbb{R}}

\newcommand{\bbem}{\begin{bem}}
\newcommand{\ebem}{\end{bem}}
\newcommand{\bbez}{\begin{bez}}
\newcommand{\ebez}{\end{bez}}
\newcommand{\bbsp}{\begin{bsp}}
\newcommand{\ebsp}{\end{bsp}}

\newcommand{\D}{\Delta}

\newcommand{\wt}{\widetilde}
\newcommand{\tnab}{\widetilde{\nabla}}
\newcommand{\tem}{\widetilde{M}}
\newcommand{\tg}{\widetilde{g}}

\newcommand{\cE}{{\cal E}}

\theoremstyle{definition}
\newtheorem{de}{Definition}
\newtheorem{bem}{Remark}
\newtheorem{bez}{Notation}
\newtheorem{bsp}{Example}
\theoremstyle{plain}
\newtheorem{lem}{Lemma}
\newtheorem{satz}{Proposition}
\newtheorem{folg}{Corollary}
\newtheorem{theo}{Theorem}

\newcommand{\tM}{\widetilde{M}}
\newcommand{\lu}{\underline{\lambda}}

\newcommand{\cL}{{\cal L}}
\newcommand{\cS}{{\cal S}}

\begin{document}

\bibliographystyle{abbrv}


\title{On the full holonomy group of special Lorentzian manifolds}
\author{Helga Baum}
\author{Kordian L\"arz} \address[Baum \& L\"arz]{Humboldt-Universit\"{a}t Berlin,
Institut f\"{u}r Mathematik, Rudower Chaussee 25, 12489 Berlin, Germany}\email{baum@math.hu-berlin.de \& laerz@math.hu-berlin.de}
\author{Thomas Leistner}\address[Leistner]{School of Mathematical Sciences, University of Adelaide, SA 5005,
Australia} \email{thomas.leistner@adelaide.edu.au}
\thanks{This work was supported by the Group of Eight Australia and the German Academic Exchange
Service through the Go8 Germany Joint Research Co-operation Scheme
grant ``Spinor field equations in global Lorentzian geometry''. The
third author acknowledges support from the Australian Research
Council via the grant FT110100429, the first one from the DFG-CRC
647 ``Space-Time-Matter''.}
 \begin{abstract}
 We study the full holonomy group of Lorentzian manifolds with a parallel null line bundle.
 We prove several results that are based
 on the classification of the restricted holonomy groups of such manifolds and provide a
 construction method for manifolds with disconnected holonomy which starts from a Riemannian manifold and
 a properly discontinuous group of isometries.
Most of our examples are quotients of
 pp-waves
with disconnected holonomy and without parallel vector field.
Furthermore,  we classify the full holonomy groups of solvable Lorentzian symmetric spaces and
of Lorentzian manifolds with a parallel null spinor.
Finally, we construct  examples of globally hyperbolic manifolds with
complete spacelike Cauchy hypersurfaces, disconnected full holonomy
and a parallel spinor.
\\[2mm]
\end{abstract}
\setcounter{tocdepth}{1}

\maketitle

\tableofcontents

\section{Introduction}

The aim of this paper is to
study the holonomy group of Lorentzian
manifolds with a parallel bundle of null lines. The holonomy group
of a semi-Riemannian manifold $(M,g)$ at a point $p\in M$  is given
as the group of parallel transports along loops\footnote{All curves we consider
are piecewise smooth.} based at $p$,
\begin{equation}
\label{holdef} \Hol_p(M,g):= \{ P_\gamma:T_pM\to T_pM\mid
\gamma:[0,1]\to M \text{ a  curve {} with } \gamma(0)=\gamma(1)=p\}.
\end{equation}
Here $P_\gamma$ denotes the parallel transport along $\gamma$ with
respect to the Levi-Civita  connection $\nabla^g$ of $g$. The holonomy group is a subgroup
of the orthogonal group $\O(T_pM,g_p)$, where $T_pM$ is the tangent space of $M$ at $p$ and $g_p$
the scalar product induced by the semi-Riemannian metric $g$ on $T_pM$.

Holonomy groups are not necessarily closed nor connected. The
connected component $\Hol^0_p(M,g)$ is given by restricting the
definition \eqref{holdef} to curves that can be  contracted to  the
point $p$. Indeed, by contracting the loop $\gamma$ we obtain a path
in the  holonomy group from $P_\gamma$ to the identity. Hence,
holonomy groups of simply connected manifolds are connected. The
{\em restricted} holonomy group $\Hol_p^0(M,g)$ is a normal subgroup
in the {\em full} holonomy group $\Hol_{p}(M,g)$. Moreover,  the
fundamental group of $M$ surjects homomorphically onto their
quotient,
\begin{equation}\label{fundamental}
\begin{array}{rcl}
    \pi_{1}(M,p) & \twoheadrightarrow & \Hol_{p}(M,g)/\Hol^{0}_{p}(M,g)\\
        \left[\gamma\right] &\mapsto &  \left[P_\gamma\right]
\end{array}\end{equation}
(for a proof see for example \cite[Chap. II, Sec. 4]{kob/nom63}).
The holonomy group is a very powerful tool, for example, for
determining parallel sections in geometric vector bundles. Knowing
the holonomy group of a given semi-Riemannian manifold allows to
find the solution to the partial differential equation for a
parallel section by solving an algebraic problem, namely to
determine the fixed vectors of the corresponding representation of
the holonomy  group. The parallel section is then obtained by
parallel transporting the algebraic object at a point to the whole
manifold and thus defining a global section. For example, for
finding a parallel vector field one has to find a fixed vector under
the holonomy group acting on the tangent space. For finding a
parallel spinor field, one has to find a spinor that is fixed under
the image of the holonomy group
 in the corresponding spinor group, or for simply connected manifolds, one that is fixed under the spin representation of the holonomy algebra.
 Another important example is the parallel complex structure of a K\"ahler manifold. Here the  the holonomy group is reduced to the unitary group.
These facts also show the importance of manifolds with special holonomy, e.g. Calabi-Yau manifolds, in string theory,
where in some situations the underlying spacetime is required to have a covariantly constant, i.e. parallel,  spinor field.
For these reasons, a classification of possible holonomy groups of semi-Riemannian manifolds is a desirable result, but out of reach in full generality.

 Classification results for holonomy groups are usually obtained
only  for the restricted holonomy group (see for example
\cite{berger55,schwachhoefer1,leistnerjdg}).  Such results are based on the
Ambrose-Singer holonomy theorem \cite{as} which states that the Lie
algebra of the holonomy group is generated by curvature at every point in $M$.  More precisely, the Lie algebra of the holonomy group at $p$  is generated as a vector space  by the following linear maps of $T_pM$
\begin{equation}
\label{ambrose}
P_\gamma^{-1}\circ R_{\gamma(1)}(X,Y)\circ P_\gamma,\end{equation}
where $\gamma:[0,1]\to M$ is a curve starting at $p$, $R_{\gamma (1)} $ the curvature tensor at $\gamma(1)$, and $X,Y\in T_{\gamma(1)}M$.

  Since the
Levi-Civita connection is torsion free, the  curvature and hence all the maps in \eqref{ambrose},  satisfy the
Bianchi identities. Hence, via the Ambrose-Singer holonomy theorem,
the Bianchi identities impose strong algebraic conditions on the Lie
algebra of the holonomy group which lead to
classification results, but only for the restricted holonomy groups,
and mostly under the assumption that it acts irreducibly, or at least indecomposably (see next paragraph for the definition).
Similar
classification results  for  full holonomy groups  are out of reach. For example,  due to the complete reducibility of the holonomy representation for Riemannian manifolds, the restricted holonomy group of a Riemannian manifold is always closed and hence compact. But  the
example given in \cite{wilking99} shows that $\Hol_{p}(M,g)$ can be
non-compact even for {\em compact} Riemannian manifolds.

For Lorentzian manifolds the classification of restricted holonomy
groups is obtained as follows: Using the splitting theorems by de
Rham \cite{derham52} and Wu \cite{wu64} one can decompose every
simply connected, complete Lorentzian manifold into a product of
Riemannian manifolds and a Lorentzian manifold, all simply
connected and complete,  and with indecomposably acting holonomy
group. By indecomposable we mean that the metric degenerates on every subspace of $T_pM$ that is invariant under the holonomy group. Of course, for Riemannian
manifolds this implies that the holonomy group acts irreducibly and
one can apply Berger's holonomy classification \cite{berger55} to
the Riemannian factors. The remaining Lorentzian factor is either
flat, irreducible or indecomposable.
On the one hand,  the irreducible case is dealt with by the Berger's list  \cite{berger55}, on which    $\SO^0(1,n-1)$ is the only possible irreducible restricted holonomy group of Lorentzian manifolds. This also follows from the more fundamental result by Di Scala and Olmos in  \cite{olmos-discala01} that $\SO^0(1,n-1)$ has no proper irreducible subgroups.
On the other hand,  the classification in the  indecomposable, non-irreducible case was achieved recently  by Berard-Bergery and Ikemakhen \cite{bb-ike93},  the third author \cite{leistnerjdg}, and Galaev \cite{galaev05}. We will explain the classification in the following paragraph and in Section \ref{alg-sec}.

We say that a Lorentzian manifold $(M,g)$ of dimension $(n+2)$ {\em has  special holonomy}, or simply is {\em special} if its restricted holonomy group acts indecomposably but is not equal to $\SO^0(1,n+1)$, the connected component of the special orthogonal group in Lorentzian signature\footnote{This is in accordance with the terminology in the Riemannian setting where {\em special holonomy} usually refers to manifolds with restricted holonomy different from $\SO(n)$ but still acting irreducibly.  Of course, in Riemannian signature  irreducibility is the same as indecomposability, but in other signatures indecomposability is the property that is geometrically more important: If the restricted holonomy group acts decomposably, i.e. with a non-degenerate invariant subspace, then, without further assumptions on $M$,  the manifold is locally a semi-Riemannian product \cite[Proposition 3]{wu64}.}.
As $\SO^0(1,n+1)$ has no proper irreducible subgroups,
  this means that the representation of the restricted holonomy group of a special Lorentzian manifold cannot be irreducible but that
 the metric is degenerate on all invariant subspaces. Hence, there is a $\Hol^{0}_{p}(M,g)$-invariant degenerate
 subspace $W \subset T_{p}M$ which defines an invariant null line $L:=W\cap W^\bot$ and the restricted holonomy
 group is contained in the stabiliser in $\SO^0(T_pM,g_p)$ of this line $L$.
Identifying $T_pM$ with $\rr^{1,n+1}$ by fixing a basis  $(\ell ,e_{1},\ldots,e_{n}, \ell^*)$ in
$T_pM$ such that $\ell\in L$ and the metric at $p$ is of the form
\begin{equation}\label{metric}
\begin{pmatrix}
0&0&1\\0&\1_n & 0 \\1&0&0\end{pmatrix},
\end{equation}
this stabiliser in $\O(1,n+1)$ can be written as the parabolic subgroup
\be
P& :=&  \mbox{Stab}_{\O(1,n+1)}(L) \ =\ (\rr^*\times \O(n))\ltimes \rrn
    \\[2mm]
    &=&
         \left\{ \begin{pmatrix}
                                    a & x^t &
                                    -\frac{1}{2}a^{-1}x^{t}x \\[0.1cm]
                                    0 & A & -a^{-1}Ax \\[0.1cm]
                                    0 & 0 & a^{-1}
                                \end{pmatrix} \;\;
                                \Big| \;\; a \in \bbR^{*},\ A \in \O(n),
                                \ x \in \bbR^{n} \right\},
\ee
whose connected component is given by the stabiliser
of $L$ in $\SO^0(1,n+1)$, i.e.,
\[P^0\ =\   \mbox{Stab}_{\SO^0(1,n+1)}(L)=(\rr^+\times \SO(n))\ltimes \rrn.\]
This defines three projections of $P$  onto $\rr^*$, $\O(n)$ and
$\rr^n$, and of $P^0$ onto $\rr^+$, $\SO(n)$ and $\rr^n$ which we
denote by \be \pr_\rr:P\to \rr^*, & \;\;  \pr_{\O(n)}:P\to \O(n), &
\;\;  \pr_{\rrn}:P\to \rrn.\ee


For Lorentzian manifolds with restricted holonomy group $H^0$ acting indecomposably but not irreducibly in \cite{leistnerjdg}
it was shown that
$\pr_{\O(n)}(H^0)\subset \SO(n)$ has to be the holonomy group of a Riemannian manifold. Using
 results in \cite[see our Section \ref{alg-sec}]{bb-ike93}  this gave a full classification of restricted holonomy groups of
 Lorentzian manifolds acting indecomposably and not irreducibly.
 Galaev \cite{galaev05} then extended previous existence results and verified that indeed all groups on the list can
be realised as holonomy groups of Lorentzian manifolds. Together
with the splitting theorems by  de Rham \cite{derham52} and Wu
\cite{wu64} and the fact mentioned above that $\SO^0(1,n+1)$ has no proper
irreducible subgroups, this yields the classification of restricted
holonomy groups of Lorentzian manifolds.

Our first result about the full holonomy group is that it has the same $\rrn$-part as the restricted holonomy (see Proposition
\ref{types} for a more precise statement):
\btheo\label{theo1} Let
$(M,g)$ be a Lorentzian manifold of dimension $(n+2)>2$ such that
its restricted holonomy group $\Hol_p^0(M,g)$ acts
 indecomposably but not irreducibly. Then
\bnum
\item[1)]
the full holonomy group $\Hol_p(M,g)$ acts  indecomposably but not irreducibly, and
\item[2)]
there is a subset $\Gamma\subset   \rr^*\times \O(n)$ such that
\[ \Hol(M,g)=\Gamma\cdot \Hol^0(M,g).\]
\enum \etheo After recalling the basics on special Lorentzian
geometry and proving this result, a large part of the paper is
devoted to the construction of Lorentzian manifolds with
disconnected holonomy. Our construction uses a method to obtain
special Lorentzian manifolds of dimension $(n+2)$ from Riemannian
manifolds of dimension $n$. Using this method for every group $G$
that is a Riemannian holonomy group, connected or disconnected, we
obtain special Lorentzian manifolds with holonomy
\[
G\ltimes \R^n,\ (\R^+ \times G)\ltimes \R^n,\ (\R^* \times G)\ltimes
\R^n,\ (\Z_2\times G)\ltimes \rrn\] (see Proposition
\ref{satz-realisation-standard} for details). Further examples are
obtained as quotients of Lorentzian manifolds by a properly
discontinuous group of isometries $\Gamma$. To this end, in
Proposition~\ref{coversatz} of Section \ref{coverings} we prove a
generalisation of the fundamental formula \eqref{fundamental} for
general coverings  $\pi:\tem\to M:=\tem/\Gamma $, which provides a
surjective group homomorphism \be
\Gamma&\twoheadrightarrow&\Hol_p(M)/\Hol_{\wt p}(\tem)
\\
\s&\mapsto& \left[ P_{\gamma}\right], \ee where $\gamma $ is a loop
at $p$ that, when lifted to a curve $\wt\gamma$ starting at $\wt p$,
ends at $\s^{-1}(\wt p)$,
and yields a formula of the parallel transport in $M$ in terms of that in $\tem$.
Applying this to our context in Theorem \ref{method} and Corollary \ref{method-folge}
leads to a variety of examples of special Lorentzian manifolds with disconnected holonomy
 in Section \ref{section-pp-waves} including an example for which the quotient $\Hol/\Hol^0$ is infinitely generated.
These various examples illustrate the
possible differences between the full and the restricted holonomy group and feature a coupling between the
$\O(n)$-part and the $\rr^*$-part of the holonomy group that is not present for the restricted holonomy group.
Most of our examples are quotients of pp-waves.

One of the class of examples we consider are solvable Lorentzian symmetric spaces,
so-called Cahen-Wallach spaces \cite{Cahen-Wallach:70,Cahen-Parker:80}.
In Proposition \ref{satz-fullhol-symm}  we show that the full holonomy
of a Cahen Wallach space is either connected, in which case it is equal to
$\rrn$, or given as $
\Z_2 \ltimes \R^n$, where the $\Z_2$ factor is generated by a reflection in $\O(n)$.

In the last part of the paper we consider the full holonomy of
Lorentzian manifolds that admit a parallel spinor field.  Such a spinor
field induces a parallel vector field and hence the holonomy
stabilises  a vector, which, for indecomposable manifolds, has to be null, i.e. lightlike.
First we show in Proposition \ref{satz-Ex-paral-spin} and Corollary \ref{Folg-Number-parspinor} that,
for  a time- and space-orientable Lorentzian manifold with holonomy $G\ltimes \rrn$, the existence of
a spin structure with  parallel spinors  depends solely on $G$.
This result enables us to apply to the Lorentzian situation
 the
classification of irreducible subgroups of $\O(n)$ stabilising a
spinor and having $\SU(\frac{n}{2})$, $\Sp(\frac{n}{4})$, $\G_2$ or
$\mathrm{Spin}(7)$ as connected component. This classification was
given by McInnes  \cite{mcinnes91} and Wang \cite{wang95} and it yields our
 \btheo\label{theo2} Let $(M,g)$ be a Lorentzian spin
manifold of dimension $(n+2)>2$ with full holonomy group
$H=\Hol_p(M,g)$ with a parallel spinor. Assume that \bnum
\item[(i)] the connected component $H^0$ of the holonomy group $H$ acts indecomposably, and
\item[(ii)] $G^0:=\pr_{\O(n)}(H^0)$ acts irreducibly on $\rrn$.
\enum
Then $H=
G\ltimes \rr^{n}$,
where $G\subset \SO(n)$
is one of the groups listed in Theorem \ref{wangtheo} and the dimension of parallel spinors on $(M,g)$
is equal to the dimension  $N$ of spinors fixed under $G$ as given in Theorem \ref{wangtheo}.
\etheo Finally, we study the existence problem for metrics with these
holonomy groups and parallel spinors. We use a method developed in
\cite{baum-mueller08} to construct globally hyperbolic Lorentzian
manifolds with complete spacelike Cauchy hypersurfaces,  parallel
spinors and holonomy $G\ltimes\rr^n$ from Riemannian manifolds. We
apply this method to examples given by Moroianu and Semmelmann in
\cite{moroianu-semmelmann00} and obtain globally hyperbolic metrics
with parallel spinors for the groups in Theorem \ref{theo2}.

\section{Algebraic preliminaries}
\label{alg-sec}

Let $(M,g)$ be a Lorentzian manifold of signature $(1,n+1)$. The
holonomy group as defined in \eqref{holdef} is an immersed Lie
subgroup of $\O(T_pM,g_p)$ (for a proof see  \cite[Thm. II.4.2]{kob/nom63}). We
denote its Lie algebra  by $\hol_p(M,g)$. For connected manifolds,
holonomy groups at different points are conjugated in $\O(1,n+1)$ to
each other. We assume from now on that all manifolds are connected.
Hence we may omit the point $p$ and consider holonomy groups only up
to conjugation.

We will first derive some purely algebraic results which will imply
Theorem \ref{theo1} of the introduction. Let $\rr^{1,n+1}$ be the
$(n+2)$-dimensional Minkowski space, in which we fix a basis
  $(\ell,e_{1},\ldots,e_{n},\ell^*)$ such that  the Minkowski inner product is of the form \eqref{metric}.
Let $L$ be the null line spanned by $\ell$. Furthermore let
$H\subset \O(1,n+1)$ be a subgroup and $H^0$ a normal subgroup of
$H$. Obviously,  $H$ is contained in the normaliser in $\O(1,n+1)$
of $H^0$,
\[H\subset \Nor_{\O(1,n+1)}(H^0).\]
In this situation we prove:

\blem\label{lem1} Let $H^0\subset \O(1,n+1)$ be a subgroup that acts
indecomposably and stabilises the null line $L$. Then the normaliser
of $H^0$ stabilises $L$ as well, i.e.
\[ \Nor_{\O(1,n+1)}(H^0)\subset \Stab_{\O(1,n+1)}(L). \]
In particular, if $H\subset \O(1,n+1)$ is an immersed Lie group and
$H^0$ is the connected component of $H$, then $H$ stabilises $L$ and
acts indecomposably if $H^0$ does. \elem \bprf Let $g\in
\Nor_{\O(1,n+1)}(H^0)$ and $L=\rr\cdot \ell$. Then for each $h\in
H^0$ we have that also $\hat h:=ghg^{-1}\in H^0$. Since $H^0$
stabilises $L$ there are $\hat{\lambda}$  such that
 $\hat{h}(\ell)=\hat \lambda \ell$. Multiplying this with $g^{-1}$ gives $hg^{-1}(\ell)=\hat\lambda g^{-1}(\ell)$.
 Hence, $g^{-1}(\ell)$ spans a null line that is fixed under all of $H^0$. But since $H^0$ was assumed to be indecomposable,
 $L$ is the only line that is fixed by $H^0$. Hence, $g^{-1}(\ell)\in L$.
\eprf
\bbem\label{bem1} We should remark that we can have immersed
subgroups that fix $L$ and acting  indecomposably, but the connected
component  $H^0$ acts decomposably. An example of this is given in
Section \ref{section-pp-waves}. \ebem

From now on let $H\subset \O(1,n+1)$ be an immersed subgroup with connected component
\[H^0\subset P^0= \mathrm{Stab}_{\SO^0(1,n+1)}(L)=(\rr^+\times \SO(n))\ltimes \rr^n\] in the
stabiliser of a null line $L$.  Lemma \ref{lem1} ensures that
\[
H\subset P= \Stab_{\O(1,n+1)}(L)=(\rr^*\times \O(n))\ltimes \rr^n\] and we can define
\[G:=\pr_{\O(n)}(H).\]
If $G^0$ denotes the connected component of $G$ we have
\[G^0=\pr_{\O(n)}(H^0)=\pr_{\SO(n)}(H^0).\]
We denote by $\h\subset \so (1,n+1)$ the Lie algebra of $H^0$ and
recall the classification of subalgebras of $\so(1,n+1)$ that act
indecomposably but not irreducibly given in \cite{bb-ike93}.
 If $\lah$ is such a subalgebra, then it is contained in the Lie algebra of the stabiliser $\mf p$ of
 the null line $L$, i.e. $\h \subset \mf p:=(\rr \+\so(n))\ltimes \rrn$. We will write elements in
 $\mf p$ as triple $(a,X,v)\in \rr\times \so(n)\times \rrn$.  Denote by $\g$ the projection of $\h$
 onto $\so(n)$. Since $\g$ is reductive, it decomposes into its centre $\mf{z}$ and its derived Lie
 algebra $\g'$, i.e. $\g=\mf{z}\+\g'$. Then it was proven in \cite{bb-ike93} that, if $\h$ acts
 indecomposable, it is of one of the following types, the first two being {\em uncoupled} and the last two {\em coupled}:
\begin{description}
\item[Type 1] $\h=(\rr\+\g)\ltimes \rrn$,
\item[Type 2]  $\h=\g\ltimes \rrn$,
\item[Type 3]
There exists an epimorphism $\varphi: {\mathfrak z} \rightarrow \rr $, such that
${\mathfrak h} = \left( {\mathfrak f} \oplus \mf{g}' \right)\ltimes \rrn,$
where ${\mathfrak f}:= \mathrm{graph}\ \varphi = \{ ( \varphi(Z),Z) | Z\in {\mathfrak z} \} \subset \rr\oplus{\mathfrak z}  $. Or, written in
matrix form:
\[{\mathfrak h}= \left\{ \left.
\left(
\begin{array}{ccc}
 \varphi(Z)     &   v^t & 0\\
  0 &Z+X& -v\\
0&0& -\varphi(Z)\\
\end{array}
\right) \right| Z\in {\mathfrak z} , X\in \mf{g}', v\in \mathbb{R}^n\right\}.
\]
\item[Type 4]
 There exists
 a decomposition $\rrn= \mathbb{R}^k\oplus \mathbb{R}^{n-k}$, $0<k<n$, and
 an epimorphism $\psi: {\mathfrak z} \rightarrow \rr^k$,
such that   $ {\mathfrak h} =  \left(  {\mathfrak f} \oplus
\mf{g}'\right)\ltimes \rr^{n-k} $ where ${\mathfrak f}:= \{ \left( Z
, \psi(Z)\right) | Z\in {\mathfrak z} \} = \mathrm{graph}\ \psi
\subset\mf z\+ \rr^k $. Or, written in matrix form:
\[{\mathfrak h}= \left\{ \left.
\left(
\begin{array}{cccc}
0    & \psi(Z)^t    &   v^t & 0\\ 0    &  0   &0& -\psi(Z)\\ 0&0&Z+X&-v \\ 0&0&0&0
\end{array}
\right) \right| Z\in {\mathfrak z} , X\in \mf{g}', v\in \mathbb{R}^{n-k}
\right\}.
\]
\end{description}
Note that, since $\h$ acts indecomposably, its projection onto $\rrn$ is always all of $\rrn$,
\[\pr_{\rrn}(\h)=\rrn,\]
for all four types. However,  in the second coupled type, $\h$ does
{\em not} contain $\rrn$, only $\rr^{n-k}$. The connected Lie groups
$H^0$ corresponding to $\h$ of types 1 and 2 are given as
 \[( \rr^+\times G^0)\ltimes \rrn\  \text { or }\ G^0\ltimes \rrn.\]
 Denote by $G'^0 $ the connected subgroup in $P$ that corresponds to the derived Lie algebra $\g'$ of $\g$.
For the  coupled type 3 we have that
\[H^0=(F^0\times G'^0)\ltimes \rrn,\]
where $F^0$ is the connected Lie group corresponding to the Lie algebra
${\mathfrak f}= \{ ( \varphi(Z),Z) | Z\in {\mathfrak z} \} \subset \rr\oplus{\mathfrak z} $.
For the last coupled type the connected component of $H$ is given by
\[H^0= (F^0\times G'^0)\ltimes \rr^{n-k},\]
where $F^0$ is the connected Lie group corresponding to the Lie
algebra ${\mathfrak f}= \{ \left(  Z, \psi(Z) \right) | Z\in {\mathfrak z}
\} \subset \mf z\+ \rr^k $. In all cases we have that
$\pr_{\rrn}(H^0)=\rrn$ and $G^0:=\pr_{\SO(n)}(H^0)$  is given by the
the connected Lie subgroup in $\SO(n)$ corresponding to $\lag$.

\bs\label{Satz-full-hol}\label{types} Let $H^0\subset \SO^0(1,n+1)$
be the connected component of an immersed Lie group $H\subset
\O(1,n+1)$, and assume that $H^0$ acts indecomposably and not
irreducibly. If $G:=\pr_{\O(n)}(H)\subset \O(n)$ is the projection
of $H$ onto $\O(n)$ and $G^0=\pr_{\O(n)}(H^0)\subset \SO(n)$ its
connected component, then,  for the four different types, we have:
\begin{description}
\item[Type 1] $H=(\rr^*\times G) \ltimes \rrn$, or $H=(\rr^+\times G)\ltimes \rrn$,
\item[Type 2] $H=\hat{G} \ltimes \rrn$, where $\hat{G}\subset \rr^*\times G$ with connected component $G^0$,
\item[Type 3] There is a  subset $\Gamma\subset\Z_2\times G\subset \rr^*\times \O(n)$ such that $H=\Gamma \cdot H_0$.
\item[Type 4] There is a  subset $\Gamma\subset \rr^*\times G$ such that $H=\Gamma \cdot H_0$.
\end{description}
\es

\bprf In order to prove the statement, we show in  all four cases,
\begin{equation}\label{forall}
\text{for every $P\in H$ there is an element $Q\in H^0$ such that $P\cdot Q\in \rr^*\times \O(n)$.}
\end{equation}
For the first three types for which we have $\rrn\subset H^0\subset
H$ the statement is obvious: Here we have \be P&=&\begin{pmatrix} a
& v^t& * \\0&A&* \\ 0& 0&a^{-1}\end{pmatrix}\in H \ee
and we find
\[
Q := \exp \begin{pmatrix} 0 & - a^{-1}v^t& 0 \\0&0&a^{-1}v \\ 0&
0&0\end{pmatrix} = \begin{pmatrix} 1 & - a^{-1}v^t& * \\0&\1&* \\ 0&
0&1\end{pmatrix} \in\rr^n\subset H^0\] such that
\[P\cdot Q
=
\begin{pmatrix} a & 0& 0 \\0&A&0 \\ 0& 0&a^{-1}\end{pmatrix}
\in \rr^*\times \O(n).
\]
This implies the form of $H$ in the types 1 and 2. For type 3 we can
prove more. For an arbitrary
\[
P=\begin{pmatrix} \pm\e^{a} & v^t& * \\0&A&* \\ 0&
0&\pm\e^{-a}\end{pmatrix}\in H
\] we choose $Z\in \mf z$ such that $a+\vf(Z)=0$ and consider
\[
Q_1  =: \exp \begin{pmatrix} \vf(Z) & 0& 0 \\0&Z&0 \\
0& 0&-\vf(Z)\end{pmatrix} =\begin{pmatrix}  \e^{\vf(Z)}& 0& *
\\0&\exp(Z)&* \\ 0& 0&\e^{-\vf(Z)}\end{pmatrix} \in F^0\subset H^0\]
and
\[ Q_2 := \begin{pmatrix} 1 & \mp v^t e^{Z} & * \\
0 & \1 & * \\ 0 &  0 & 1 \end{pmatrix}  \in \R^n \subset H^0.\] Then
\[
P\cdot Q_1\cdot Q_2= \begin{pmatrix} \pm\e^{a+\varphi(Z)} & v^t e^Z& * \\0&A e^Z&* \\
0& 0&\pm\e^{-(a+\varphi(Z))}\end{pmatrix}  \begin{pmatrix} 1 &
\mp v^t e^Z& *
\\0&\1&* \\ 0& 0&1\end{pmatrix} =
\begin{pmatrix} \pm 1 & 0& * \\0& A e^Z &0 \\ 0& 0&\pm 1\end{pmatrix}
\in \Z_2\times \O(n)\,.
\]

\noindent In the case, when $\rr^n\not\subset H^0$, the Lie algebra
$\h$ of $H^0$ is of the second coupled type and we write
\[P=
\begin{pmatrix} a & u^t&v^t& * \\0&A&B&* \\0&C&D&*\\0& 0& 0&a^{-1}\end{pmatrix}\in H,
\]
with $u\in \rr^k$ and $v\in \rr^{n-k}$. Since the linear map $\psi:
\mathrm{pr}_{\so(n)}(\h) \to \rr^k$ is surjective, we find an $X\in
\so(n-k)$ such that $\psi(X)=-a^{-1}u$. Then for \be Q_1&:=&\exp
\begin{pmatrix} 0 & 0&-a^{-1}v^t& 0 \\0&0 &0&0\\0&0&0 &a^{-1}v\\0&
0& 0&0\end{pmatrix}\
=\ \begin{pmatrix} 1 & 0&-a^{-1}v& * \\0&\1 &0&* \\0&0&\1 &*\\0& 0& 0&1\end{pmatrix}\ \in\  H^0\\
Q_2&:=& \exp
\begin{pmatrix} 1& \psi(X)^t&0& 0\\0&0&0&-\psi(X) \\0&0&X&0\\0& 0& 0&0\end{pmatrix}
\ =\
\begin{pmatrix} 1& \psi(X)^t&0& * \\0&1&0&* \\0&0&\exp(X)&*\\0& 0& 0&1\end{pmatrix}
\ \in\  H^0
\ee
we obtain
\[
P \cdot Q_1 \cdot Q_2 =
\begin{pmatrix} a & u^t&0& * \\0&A&B&* \\0&C&D&*\\0& 0& 0&a^{-1}\end{pmatrix}
\cdot
Q_2
\\
=
\begin{pmatrix} a & a\psi(X)^t+u^t&0& * \\0&A&B\exp(X)&* \\0&C&D\exp(X)&*\\0& 0& 0&a^{-1}\end{pmatrix}
\in \rr^*\times \O(n),
\]
as $\psi(X)=-a^{-1}u$. This verifies \eqref{forall} also for type 4
and proves the proposition. \eprf

\noindent Note that Lemma \ref{lem1} and Proposition
\ref{Satz-full-hol} imply Theorem \ref{theo1} from the introduction
when applied to the full and the restricted holonomy group of a
Lorentzian manifold.

\section{Null line bundle and screen bundle}

In this section let $\Hol_p(M,g)$ and $\Hol^0_p(M,g)$ be the full
and the restricted holonomy group of a Lorentzian manifold $(M,g)$
of dimension $n+2>2$, and let $\nabla$ denote the Levi-Civita
connection of $g$. We assume that the restricted holonomy group acts
indecomposably and not irreducibly. Then, from Theorem \ref{theo1}
we know that the same holds true for the full holonomy group. Hence,
by the fundamental principle of holonomy by which holonomy invariant
subspaces  correspond to distributions on the manifold that are
invariant under parallel transport \cite[10.19]{besse87}, the
manifold admits a global distribution ${\cal L}$ of null lines that
is invariant under parallel transport. Of course, also the
distribution ${\cal L}^\bot$ whose fibres are orthogonal to the
fibres of ${\cal L}$ is invariant under parallel transport. Hence,
the tangent bundle is filtrated by parallel distributions
\[ {\cal L} \; \subset \; {\cal L}^\bot \; \subset \; TM.\]
The Levi-Civita connection induces a linear connection $\nabla^{\cal
L}$ on the bundle ${\cal L}$ by $\nabla^{\cal L}:=\pr_{{\cal L}}
\circ \nabla |_{{\cal L}}$, where $\pr_{{\cal L}}$ is the projection
onto $\cal L$. Moreover, the metric $g$ and the Levi-Civita
connection $\nabla$ induce a bundle metric $g^\cal S$ as well as a
covariant derivative $\nabla^{\cal S} $ on the so-called {\em screen
bundle}
\[
\mathcal{S}\ :=\  {\cal L}^\bot/{\cal L} \ \to \ M\] by $g^\cal
S([X],[Y]):=g(X,Y)$ and $\nabla^{\cal S}_X[Y]:=\left[ \nabla
_XY\right]$, where $[\: .\: ]:{\cal L}^\bot\to \cal S={\cal
L}^\bot/{\cal L}$ denotes the canonical projection. The following
Proposition shows the relation between the holonomy groups of
$(\cL,\nabla^{\cL})$ and $(\cS,\nabla^{\cS})$ and the projections of
$\Hol_p(M,g)$ onto $\R^*$ and $\O(n)$, respectively.

\begin{satz}
Let $(M,g)$ be a Lorentzian manifold with indecomposably,
non-irreducibly acting  restricted holonomy group, let ${\cal L}$ be
the corresponding distribution of null lines and
$\cS=\cL^{\bot}/\cL$ the sreen bundle on $M$. Then:
\begin{enumerate}
\item[1)] \label{eins} $\Hol_p(\cal L, \nabla^{\cal
L})=\pr_{\rr^*}(\Hol_p(M,g))$.
\item[2)] \label{zwei} The line bundle $\cL$ is orientable, i.e., $\cL$ admits a global
nowhere vanishing section if, and only if, $pr_{\bbR}(\Hol_{p}(M,g)) \subset
\bbR^+ $. This is equivalent to time-orientability of $(M,g)$.
\item[3)] \label{drei}  The connection $\nabla^{\cL}$ is flat if, and only if, $\pr_{\R^*}(\Hol^0_p(M,g))
=\{1\}$, and the line bundle $\cL$ has a global parallel section iff
$\pr_{\R^*}(\Hol_p(M,g))=\{1\}$.
\item[4)] \label{four} $Hol_p(\cS,\nabla^{\cS}) =
\pr_{\O(n)}(\Hol_{p}(M,g))$ and $Hol^0_p(\cS,\nabla^{\cS}) =
\pr_{\O(n)}(\Hol^0_{p}(M,g))$.
\end{enumerate}
\end{satz}
\begin{proof} The first statement is obvious since the parallelity
of $\cL$ implies $P^{\cL}_{\gamma}=P^g_{\gamma}|_{\cL}$ for any
curve $\gamma$, where $P_{\gamma}^{\cL}$ is the parallel transport
in $(\cL,\nabla^{\cL})$ and $P_{\gamma}^g$ that of $(M,g)$.

If $\cL$ admits a global section $V \in \Gamma(\cL)$, then
$P^{\cL}_{\gamma|_t}(V(\gamma(0))) = \alpha(t) V(\gamma(t))$ for all
curves $\gamma$, where $\alpha$ is a positive function with
$\alpha(0)=1$. Hence $\pr_{\R^*}(\Hol_p(M,g)) \subset \R^+$.
Conversally, let $\pr_{\R^*}(\Hol_p(M,g)) \subset \R^+$.
        Then using the holonomy principle,
    the half-line $\R^+\ell \subset \cL_p$ provides a well defined field of directions ${\cal L}^{+} \subset
{\cal L}$. Thus, we have a covering $M=\bigcup_{k}{U_{k}}$ and local
sections
        $V_{k} \in \Gamma(U_{k},{\cal L})$ such that $V_{k}(x) \in {\cal L}_x^{+}$ for any $x\in U_k$.
        Using a partition of unity we derive a global nowhere vanishing
        section of ${\cal L}$.
The orientablility of $\cL$ is equivalent to time-orientability of
$(M,g)$. To see this, we choose a splitting $s: S \to \cL^{\perp}$
of the sequence
        \begin{equation*}
            0 \rightarrow {\cal L} \rightarrow {\cal L}^{\perp} \rightarrow \mathcal{S} \rightarrow
            0.
        \end{equation*}
  Then $\cE:=s(\mathcal{S}) \subset \cL^{\perp}$ and $\cE^{\perp} \subset TM$ is a subbundle of signature $(1,1)$ with ${\cal L}\subset \cE^{\perp}$.
Hence, the light-cone of $\cE_p^{\perp}$ at any $p \in M$ is a union
of two
        lines, one of which is given by ${\cal L_p}$. Thus, we derive a second lightlike vector field $Z \in \Gamma(M,\cE^{\perp})$
        with $g(V,Z)=1$. Then
        $\frac{1}{\sqrt{2}}(V-Z)$ is a nowhere vanishing timelike unit vector field and $(M,g)$ is time-orientable.
On the other hand, any timelike unit vector field $T$ on $(M,g)$
defines a global field of null direction $\cL^+ \subset \cL$ by
requiring $g(T,\cL^+)>0$, hence, $\cL$ is orientable.

The third statement follows from the first one and standard facts of
holonomy theory: $\nabla^{\cL}$ is flat iff
$Hol_p^0(\cL,\nabla^{\cL}) = \{1\}$, the line bundle
$(\cL,\nabla^{\cL})$ admits a global parallel section if, and only if,
$Hol_p(\cL,\nabla^{\cL)})=\{1\}$.

Finally, we proof the fourth statement. Let $\gamma:[0,1]\to M$ be a
loop around the point $p=\gamma(0)=\gamma(1)\in M$. We fix a
complement $E\subset {\cal L}^\bot_p$ of ${\cal L}_p$ that is
orthogonal to ${\cal L}_p$. Then the canonical projection $[\: .\:
]:{\cal L}_p^\bot \to \cal S_p$ becomes a linear isomorphism when
restricted to $E$. For a fixed non-vanishing vector $\ell \in \cL_p
$ we denote by $V(t)$ the parallel displacement of $\ell$ along
$\gamma$ with respect to the Levi-Civita connection $\nabla^g$ of
$g$. We have to prove that the parallel transport $P_\gamma^g$ and
the parallel transport $P_\gamma^\cal S$ with respect to
$\nabla^\cal S$ commute with the canonical projection $[\: .\: ]$,
\begin{equation}\label{transport}
\left[P^g_\gamma(e)\right]=P_\gamma^\cal S([e]),
\end{equation}
for every $e\in E$. Indeed, if we write $P^{\cS}
_{\gamma|_{[0,t]}}([e])=\left[U(t)\right]$ with $U(t)\in {\cal
L}_{\gamma(t)}^\bot$ and  $U(0)=e$, we have
\[ 0 \equiv \nabla^\cal S_{\dot\gamma(t)}[U(t)]=  \left[ \nabla^g_{\dot\gamma(t)}U(t)\right]\]
which implies that
\[ \nabla^g_{\dot\gamma(t)}U(t) = f(t)V(t)\]
with a function $f:[0,1] \to \R$.  Since $V(t)$ is parallel, the
vector field
\[ U(t) - \int_0^t f(s)ds \,\cdot V(t)
\]
 is parallel along $\gamma$ with respect to $\nabla^g$ and equals $e$
in $t=0$. Hence,
\[
\left[P_\gamma^g(e)\right] = \left[U(1) - \int_0^1 f(s)ds\cdot V(1)
\right] = \left[ U(1) \right] =  P^\cal S_\gamma([e]).\] This proves
statement (\ref{transport}) and the proposition.
\end{proof}

 Since $G^{0}_{p}:= \pr_{SO(n)}(\Hol^0_p(M,g)) \subset \SO(\cal S_p)$, as a subgroup of
$\SO(n)$, is compact, it acts completely reducible on \[\cal
S_p=V_0\+V_1\+\ldots \+V_k \] with $V_0$ trivial and $V_i$
irreducible for $i=1, \ldots , k$,  and moreover, using the
Bianchi-identity,
 it can be shown (see \cite{bb-ike93} or \cite{leistnerjdg}) that \[G^{0}_{p}=G_1\times \ldots\times G_k\]
 is a direct product of subgroups $G_i$ acting irreducibly on $V_i$ and trivial on $V_j$ for $j\not=i$.
Furthermore,  in \cite{leistnerjdg} we have shown that $G^{0}_{p}$
acts as a Riemannian holonomy representation, i.e. $G^{0}_{p}$ is
trivial or a product of the groups from the Berger list, i.e. of
\begin{equation}\label{berger}
\SO(n),\ \mathrm{U}(m),\ \SU(m), \ \Sp(k),\ \Sp(k)\cdot \Sp(1),\
\mathrm{Spin}(7), \text{ and } \G_2,\end{equation}
 and of isotropy groups of Riemannian symmetric spaces.
 In section \ref{section-hol-par} we will use this in order to prove Theorem \ref{theo2}.

\section{Holonomy groups and coverings}\label{coverings}

In the following we will consider manifolds that are given as a
quotient by a group of diffeomorphisms. We recall the following
facts (see for example  \cite[Chapter 7]{oneill83}): Let $\Gamma$ be
a group of diffeomorphisms of a smooth manifold $M$. We say that
$\Gamma$ is {\em properly discontinuous}, if \bnum
\item[(PD1)] each $p\in M$ has a neighborhood $U$, such that $\gamma(U)\cap U=\emptyset$ for all $\gamma\in \Gamma\setminus
\{1\}$ and
\item[(PD2)] two points $p$ and $q$, which  are not in the same orbit under
$\Gamma$, have neighborhoods $U_p$ and $U_q$ such that $\gamma
(U_p)\cap U_q=\emptyset$ for all $\gamma\in \Gamma$. \enum
Clearly a
properly discontinuous group acts freely on $M$. If $\Gamma$ is a
properly discontinous group of diffeomorphisms of $M$, the quotient
space $M/\Gamma$ is a smooth manifold and the projection $\pi:M\to
M/\Gamma$ is a smooth covering map \cite[Chapter 7, Proposition
7.7]{oneill83}). If $\Gamma$ is a properly discontinuous group of
isometries of a semi-Riemannian manifold $M$, then there is a unique
metric on $M/\Gamma$, such that $\pi: M \to M/\Gamma$ is a
semi-Riemannian covering (\cite[Chapter 7, Corollary
7.12]{oneill83}).

\bbem \label{Remark-properly disc} Let $\Gamma^{\Omega}_{\lambda}$,
$\lambda \in \R \setminus \{0\}$,  be the group of diffeomorphisms
of an appropriate domain $\Omega \subset \R^2$ generated by
\[ \varphi_{\lambda}(v,u) := (e^{\lambda}v, e^{-\lambda}u), \qquad
(v ,u) \in \Omega. \] For $\Omega^0:= \R^2 \setminus \{(0,0\}$, the
group $\Gamma^{\Omega^0}_{\lambda}$ fails to satisfy (PD2) for the
points $p=(1,0)$, $q=(0,1)$. For the halfspace $\Omega:= \{(v,u) \in
\R^2 \mid u>0\}$, the group $\Gamma^{\Omega}_{\lambda}$ is properly
discontinuous whereas the groups
$\Gamma^{\Omega}_{\lambda_1,\lambda_2} \subset
\mathrm{Diff}(\Omega)$, generated by $\varphi_{\lambda_{1}}$ and
$\varphi_{\lambda_{1}}$ for $\lambda_1$, $\lambda_2$ linearly
independent over $\QQ$, fail to satisfy (PD1) since there is a
sequence of integers $(k_n,l_n) \in \Z\times \Z$ such that $0\not =
k_n\lambda_1 + l_n \lambda_2 \stackrel{n\to \infty}{\longrightarrow}
0$. \ebem

In the following section we will have to check that certain group
actions on manifolds define quotients that are manifolds again. The
following simple observation is useful. Its proof is
straightforward.

\begin{de} Let $M_1$ and $M_2$ be two smooth connected manifolds
and $\Gamma \subset \mathrm{Diff}(M_1) \times \mathrm{Diff}(M_2)$ a
group of diffeomorphisms of $M_1 \times M_2$. We denote by $\Gamma_i
:= \mathrm{proj}_i(\Gamma) \subset \mathrm{Diff}(M_i)$ the
projections and for $\sigma \in \Gamma_2$ by $\Gamma_{\sigma}
\subset \Gamma_1$ the $\sigma$-section $\,\Gamma_{\sigma} := \{
\gamma \in \Gamma_1 \mid (\gamma,\sigma) \in \Gamma \}$.  We say
that $\Gamma$ is of {\em quotient type}, if
\begin{enumerate}
\item[(i)] $\Gamma_2$ is properly discontinuous,
\item[(ii)] $\Gamma_{\mathrm{Id}_{M_2}} \subset \Gamma_1$ satisfies (PD1)
and
\item[(iii)] $\Gamma_{\sigma} \subset \Gamma_1$ satisfies (PD2) for
all $\sigma \in \Gamma_2$,
\end{enumerate}
(or if the same is true exchanging the role of $M_1$ and $M_2$).
\end{de}

\noindent Clearly, (iii) is satisfied if the section sets
$\Gamma_{\sigma} \subset \Gamma_1$ are finite for all $\sigma \in
\Gamma_2$.

 \blem\label{properlemma} Let $\Gamma \subset
\mathrm{Diff}(M_1) \times \mathrm{Diff}(M_2)$ be a group of quotient
type. Then $\Gamma$ is properly discontinuous. \elem

The fundamental relation ~\eqref{fundamental} between the
fundamental group and the quotient of the full holonomy by the
restricted holonomy group is a special case of the following fact
when applied to the universal covering. In the following when
referring to semi-Riemannian manifold, for brevity of notation we do
not mention explicitly the metrics. Furthermore, $\wt P_{\wt
\gamma}$ denotes the parallel transport along a curve $\wt \gamma$
in $\tem$ and $P_\gamma$ along a curve $\gamma$ in $M$, both with
respect to the Levi-Civita connections of the semi-Riemannian
metrics on $\tem$ and $M$, respectively.

\bs\label{coversatz} Let $\tem$ be a connected semi-Riemannian
manifold and $\Gamma$ a properly discontinuous group of isometries
of $\tem$ inducing the semi-Riemannian covering $\pi:\tem\to
M:=\tem/\Gamma $. Then, for any points $p\in M$ and $\wt p$
in the fibre $\pi^{-1}(p)$ we have:
\begin{enumerate}[(i)]
\item
The holonomy group $\Hol_{\wt p}(\tem)$ injects homomorphically into $\Hol_{p}(M)$ via
\[
\iota:\wt P_{\wt \gamma}\mapsto P_{\pi\circ\wt \gamma},\]
for $\wt \gamma$ a loop at $\wt p$,
 and the image is a normal subgroup.
\item The following map is a surjective group homomorphism,
\be
\Phi:\Gamma&\to&\Hol_p(M)/\Hol_{\wt p}(\tem)
\\
\s&\mapsto& \left[ P_{\gamma}\right], \ee where $\gamma $ is a loop
at $p$ that, when lifted to a curve $\wt\gamma$ starting at $\wt p$,
ends at $\s^{-1}(\wt p)$.
\item Let $\gamma$ be a loop in $M$ at $p\in M$. Then,
using the identification of $T_{p}M$ with $T_{\wt p}\widetilde{M}$
by $d\pi_{\wt p}$, the parallel transport along $\gamma$ is given by
\begin{equation}
\label{parallel-quotient} P_{\gamma}= d\sigma_{\s^{-1}(\wt p)} \circ
\tilde{P}_{\tilde{\gamma}},
\end{equation}
where $\wt \gamma$ is the lift of $\gamma$ starting at $\wt p$ and
$\wt\gamma(1)=\s^{-1}(\wt p)$ with $\sigma\in \Gamma$. In
particular,
\begin{equation} \phi(\sigma):= d\sigma_{\s^{-1}(\wt p)} \circ
\tilde{P}_{\tilde{\gamma}} = (d\sigma^{-1}|_{p})^{-1} \circ
\tilde{P}_{\tilde{\gamma}} \label{representative-quotient}
\end{equation}
is a representative of
$\Phi(\s)\in \Hol_p(M)/\Hol_{\wt p}(\tem)$.
\end{enumerate}
 \es
\bprf
(i)
Clearly, $\iota$ is a group homomorphisms: If $\wt \gamma$ and $\wt
\delta$ are loops at $\wt p\in \tem$, we have\footnote{Note, that by
$\gamma*\delta$ we denote the joint path that first runs through
$\delta$ and then through $\gamma$.}:
\[
\iota(\wt P_{\wt \gamma}\cdot \wt P_{\wt \delta})=\iota(\wt P_{\wt
\gamma*\wt\delta}) = P_{\pi\circ (\wt \gamma*\wt\delta)} =
P_{\pi\circ \wt \gamma*\pi\circ \wt\delta} =\iota (\wt P_{\wt
\gamma})\cdot \iota(\wt P_{\wt \delta}).\] That $\iota$ is injective
follows from the fact that $\pi$ is a local
 isometry, which implies that
 \[ d\pi_{\wt{p}}^{-1} \circ P_{\pi\circ \wt \gamma} \circ
  d\pi_{\wt{p}}
  =  \wt{P}_{\wt{\gamma}}\]
  (see for example \cite[p. 91]{oneill83}).
  To show that the image of $\iota$ is a normal subgroup, we proceed in the same way as for the restricted holonomy group.
  Let $\gamma$ be a loop at $p\in M$ that lifts to a loop at $\wt p$ and $\delta$ a loop at $p=\pi(\wt p)$.
  Denote by $\wt\delta$ the lift of $\delta$ that starts at $\wt p$ and by
$\wt \gamma$ the lift of $\gamma$ starting at $\wt\delta(1)\in \pi^{-1}(p)$.
Note that $\wt \gamma$ is a loop at  $\wt\delta(1)$ and
 we have
  \[
  P_\gamma \cdot P_\delta=
P_\delta\cdot  P_{\delta^{-1} *\gamma *\delta}
  =P_\delta\cdot P_{\pi\circ (\wt\delta^{-1}*\wt\gamma*\wt\delta)}
  .\]
  Since $\wt\delta^{-1}*\wt\gamma*\wt\delta$ is a loop at $\wt p$, the image of $\iota$ is normal.

 (ii) First we have to verify that the map $\Phi$ is well defined.
 For two curves $\wt\gamma$ and $\wt\delta$  starting at $\wt p\in \tem $ and
 ending at $\s^{-1}(\wt p)$ for a $\s\in \Gamma$ we can write
 \[\wt P_{\wt\delta}=\wt P_{\wt\gamma}\cdot \wt P_{\wt\gamma^{-1}*\wt\delta}\]
 which shows that
 \[P_{\pi\circ\wt\delta}
 =P_{\pi\circ \wt\gamma}\cdot P_{\pi\circ(\wt\gamma^{-1}*\wt\delta)}.
 \]
Noting that $\wt\gamma^{-1}*\wt\delta$ is a loop at $\wt p$
shows that the image of $\s$ under $\Phi$ does not depend on the chosen curve $\wt \gamma$.

Next we show that $\Phi$ is a group homomorphism. Let $\s_1$ and
$\s_2$ two elements of $\Gamma$, and $\gamma_1$, $\gamma_2$ two
loops at $p\in M$ that lift to curves starting at $\wt p$ and ending
at $\s_1^{-1}(\wt p)$ and $\s_2^{-1}(\wt p)$, respectively.
 Now, consider the
curve $\wt \gamma:= (\s_2^{-1}\circ \wt \gamma_1) * \wt \gamma_2$ in
$\wt M$.  $\wt \gamma$ starts in $\wt p$ and ends in
$\s_2^{-1}(\s_1^{-1}(\wt p)) = (\s_1 \cdot \s_2)^{-1}(\wt p)$ and
projects to $\gamma_1*\gamma_2$. Hence, we have
\[
\Phi(\s_1) \cdot \Phi(\s_2)\ = [P_{\gamma_1*\gamma_2}] = [P_{\pi
\circ \wt \gamma}] = \Phi( \s_1 \cdot \s_2).
\]
Finally, $\Phi$ is surjective since $\tem $ is connected.

(iii)
Since $\pi$ is a local isometry, we have
       $\, P_{\gamma} = d\pi_{\tilde{\gamma}(1)} \circ \tilde{P}_{\tilde{\gamma}} \circ
        (d\pi_{\tilde{\gamma}(0)})^{-1}\,$. Then the statement follows using
the identification $T_{p}M \simeq T_{\wt p}\tM$ by $\d\pi_{\wt p}$
and $d\pi_{\s^{-1}(\wt p)}= d\pi_{\wt p} \circ d\s_{\s^{-1}(\wt
p)}$. \eprf

\section{Construction of Lorentzian manifolds with disconnected holonomy}

The following proposition shows that special classes of non-connected
subgroups of the stabilizer $\mathrm{Stab}_{\O(1,n+1)}(L)$ can
be realized as holonomy group of  Lorentzian manifolds.

\begin{satz} \label{satz-realisation-standard} Let $G\subset \O(n)$  be the holonomy group of an
$n$-dimensional Riemannian manifold. Then the subgroups
\[
G\ltimes \R^n,\
(\R^+ \times G)\ltimes \R^n,\
(\R^* \times
G)\ltimes \R^n,\ (\Z_2\times G)\ltimes \rrn\]
in $\mbox{Stab}_{\O(1,n+1)}(L)$ can be realized as holonomy groups of
Lorentzian manifolds.
\end{satz}
\begin{proof}
The first part of the proof follows \cite{leistner01} or
\cite{baum-eichfeld}, where it was given for the restricted holonomy
group. We fix a Riemannian manifold $(N,h)$ whose holonomy group is
given by the not necessarily connected group  $G$. Furthermore we
denote by $(v,u)$ coordinates on $\rr^2$ and fix an open domain
$\Omega$ in $\rr^2$, for example $\Omega=\rr^2$. Also we chose a
smooth function $f\in C^\infty (\Omega\times N )$ with the property
that
\begin{eqnarray}
\label{hess} 0 & \neq &\det \mathrm{Hess}^h_p(f(v_0,u_0,\cdot))
\quad \text{for some $p \in N$ and $(v_0,u_0) \in \Omega$.}
\end{eqnarray}
Then we define a Lorentzian manifold $(M,g)$ by
\begin{equation}\label{Lmetric}\Big( M=\Omega\times N, g^{f,h}=2\d v\d u+2f\d u^2 + h\Big).\end{equation}
Computing the Levi-Civita connection $\nabla^g$ of $g$ we obtain as
the only non-vanishing terms \be
\nabla^g_XY&=&\nabla^h_XY,\ \\
\nabla^g_{\del_u}X\ =\ \nabla^g_X{\del_u}&=& \d f(X)\del_v,
\\
\nabla^g_{\del_u}\del_u&=&\del_u(f)\del_v-\mathrm{grad}^h(f),
\\
\nabla^g_{\del_u}\del_v\ =\
\nabla^g_{\del_v}\del_u&=&\del_v(f)\del_v \ee
 for $X,Y\in \Gamma(TN)$, $\nabla^h$ and $\mathrm{grad}^h$ with respect to the Riemannian metric $h$ on $N$.
This allows us to compute the parallel transport from a point
$q=(v_0,u_0,p)$ of a vector $X_0\in T_pN$ along a curve
$\delta=(v,u,\gamma):[0,1]\to M$ for $\gamma$ a curve in $N$ and
$\delta(0)=q$. Indeed, if $X:[0,1]\to TN$ is the vector field along
$\gamma$ that is the parallel transport of $X_0\in T_pN$ with
respect to $\nabla^h$ and the function $\varphi:[0,1]\to\rr$
satisfies the ODE \be
\dot{\varphi}+\varphi \cdot \dot{u}\cdot\del_vf\circ \delta + \dot{u}\cdot \d f(X)\circ \delta &=& 0,\\
\varphi(0)&=&0, \ee then the vector field,
\begin{eqnarray} \label{parallel-general-wave}
\varphi \cdot (\del_v \circ \delta)+X
\end{eqnarray}
is the parallel transport of $X_0$ along $\delta$. This shows that
for the curve $\delta$ we have that
\[\pr_{T_pN}\circ \cal P^g_{\delta}|_{T_pN}=\cal P^h_\gamma.
\]
Hence, the $\O(n)$ projection of $\Hol_q(M,g)$ is given as
$G=\Hol_p(N,h)$.

Furthermore, when computing the curvature $R^g$ of $g$ we get \be
R^g(\del_u,X)Y&=&\mathrm{Hess}^h(f)(X,Y)\del_v. \ee Taking this at
points $(v_0,u_0,p)$ where $\det
\mathrm{Hess}^h_p(f(v_0,u_0,\cdot))\not=0$, this shows that the
holonomy algebra,  and hence both, the restricted and the full
holonomy group contain $\rrn$.

Next we look at the $\rr$-component of the holonomy group. Clearly,
when $\del_vf=0$, the vector field $\del_v$ is parallel and what we
have shown so far implies that $\Hol(M,g)=G\ltimes \rrn$. Otherwise,
again by computing the curvature term, \be
R^g(X,\del_u)\del_v&=&X(\del_v(f))\del_v \ee and by choosing  $f$
such that  $\del_vf$ is not constant on $N$, we conclude that the
holonomy algebra contains $\rr$ and hence the restricted holonomy
contains $\rr^+$. Therefore, as $(M,g)$ is clearly time-orientable,
the full holonomy group is equal $(\rr^+\times G)\ltimes \rrn$.

Using this result, we construct non-time-orientable Lorentzian
manifolds with holonomy group $(\rr^*\times G)\ltimes \rrn$ as well
as with holonomy group $(\Z_2 \times G)\ltimes \rrn$. We set $\Omega
:=\rr^2\setminus \{(0,0)\}$, and choose a function $f\in
C^\infty(\Omega\times N)$ which satisfies $f(-v,-u,p)=f(v,u,p)$ and
condition \eqref{hess}.
Then, for the
Lorentzian manifold as defined in \eqref{Lmetric}, the map
\[(v,u,p)\mapsto\sigma(v,u,p):=(-v,-u,p)
\]
is an isometry of $(M,g^{f,h})$ and generates the group $\Z_2$ in
the isometry group of $(M,g^{f,h})$. This group acts properly discontinous on $M$
and hence $M/\Z_2$ becomes a smooth Lorentzian manifold, which is
not time-orientable.
By Proposition~\ref{coversatz}, its holonomy is given as
$(\rr^*\times G)\ltimes\rrn$ if $\del_v f$ is not constant on $N$,
and as $(\Z_2\times G)\ltimes\rrn$ if $f$ is chosen such that
$\del_vf=0$. Indeed, if $\delta$ is a loop in $M/\Z_2$, starting and
ending in $\pi(0,1,p)$, with a non-closed
 lift $\wt \delta$, then $\wt\delta$  is given by $\wt\delta(t)=
(v(t), u(t), \gamma(t))$, starting at $(0,1,p)$ and  ending  at
$(0,-1,p)=\sigma^{-1}(0,1,p)$, where $\gamma$ is a loop in $N$ based
at $p$. Then, by the above formulae, the parallel transport to
$(0,-1,p)$ of the vector $\del_v$ and a vector $X_0$ tangent to $N$
at $(0,1,p)$ along $\wt \delta$ is given as $a\del_v$ for a positive
number $a$, which is $1$ when $\del_vf=0$, and by $P^h_\gamma(X_0)+b
\del v$, respectively. Hence, as
$\d\sigma_{(0,-1,p)}(\del_v)=-\del_v$ and
$\d\s_{(0,-1,p)}|_{TN}=\mathrm{Id}$,
formula~(\ref{representative-quotient}) in
Proposition~\ref{coversatz} gives the result.\end{proof}

\bbem\label{ppbem}
Lorentzian metrics of the form~\eqref{Lmetric} are sometimes  called pf-waves, for ``plane-fronted waves'' ,
or $N$p-waves, ``$N$-fronted with parallel rays'' (for example in \cite{Sanchez1} and \cite{Sanchez2}). Special classes of pf-waves are
\bnum
\item[a)]
{\em pp-waves}, ``plane fronted with parallel rays'', for which $h$ is flat,
\item[b)] {\em plane waves}, for which $h$ is flat and $f$ is a quadratic polynomial in the coordinates on $N$
with $u$-dependent coefficients, and
\item[c)] {\em Cahen-Wallach spaces}, \cite{Cahen-Wallach:70,Cahen-Parker:80},
which are Lorentzian symmetric spaces. Here  $h$ is flat and $f$ is
a quadratic polynomial in the coordinates on $N$ with constant
coefficients. \enum \ebem Note that the construction for manifolds
with holonomy $(\Z_2\times G)\ltimes \rrn$ cannot be generalised
directly to discrete subgroups of Lorentz boosts of $\rr^{1,1}$: As
we observed in remark \ref{Remark-properly disc}, a group of Lorentz
boosts $\Gamma$ generated by
$\mathrm{diag}(\e^{\lambda},\e^{-\lambda})\subset \SO(1,1)$ will not
act properly on $\rr^{1,1}\setminus \{0\}$ and the quotient will not
be Hausdorff. On the other hand, taking $\Omega$ smaller by removing
a closed ball around the origin in $\rr^{1,1}$, $\Omega$ would no
longer be invariant under $\Gamma$. We will avoid these difficulties
in the following construction by which  we will obtain Lorentzian
manifolds with disconnected holonomy group contained in the
stabiliser of a null line from the following data: \bnum
\item[1)] A simply connected Riemannian manifold $(N,h)$ together with  a function $f$ on $M$ with the property $\eqref{hess}$.
\item[2)] a properly discontinuous subgroup $\Gamma$ of isometries of the Lorentzian manifold \\
$(\tem:=\Omega\times N,g^{f,h}:=2dvdu+2fdu^2+h)$. \enum

\btheo\label{method} Let $(N,h)$ be a Riemanian manifold, $\Omega$
an open domain in $\rr^2$ with global coordinates $(v,u)$, $f$ a
smooth function on $\Omega\times N$ with $\del_vf=0$ and satisfying
property~\eqref{hess}. Then every isometry $\s$ of the Lorentzian
manifold
\[\big( \wt M=\Omega\times N, \; g^{f,h}=2\d v\d u+2f\d u^2 + h \big)\]
can be written as
\begin{equation}\label{isometry}
\sigma(v,u,p)=\Big(a_\sigma v+\tau_\sigma(v,u,p),
\frac{u}{a_\sigma}+b_\sigma, \nu_\sigma(v,u,p)\Big),\end{equation}
with certain  $a_\sigma\in \rr^*$, $b_\sigma\in \rr$,
$\tau_\sigma\in C^\infty(\tem)$ and $\nu_\sigma\in
C^\infty(\tem,N)$, such that
\[d\tau_\sigma(\del_v)=d\nu_\sigma(\del_v)=0,\]  and for each
$(v,u)\in \Omega$, $\nu_\sigma(v,u,.)$ is an isometry of the
Riemannian manifold $(N,h)$.

Furthermore, let $\Gamma$ be a properly discontinuous group of
isometries of $(\tem,g^{f,h})$ and let $\pi: \tem \to
M:=\wt{M}/{\Gamma}$ be the corresponding covering of the Lorentzian
manifold $M$.
 Then,
 at a point
 $\pi(\wt{q})\in M$
  for $\wt{q}=(v_0, u_0, p)\in \tem$,
  the map $\Phi:\Gamma\to \Hol_q(M)/\Hol_{\wt q}(\tem)$ of Proposition~\ref{coversatz}, with the image written
  in the decomposition
\[
\rr\cdot\del_v(\wt q) \+ T_{p}N \+ \rr\cdot (\del_u-f\del_v)(\wt{q})
\]
of $T_{\wt{q}}\tem\simeq T_{\pi(\wt q)}M$,
 is given by the representative
  \[
  \hat{\phi}(\s)=
\begin{pmatrix}
a_\sigma&0&0 \\ 0& (\d\nu_{\sigma^{-1}}^0|_{p})^{-1}\circ P^h_{\sigma}&0\\
0&0&a^{-1}_\sigma\end{pmatrix} \in \Phi(\sigma),
\]
where \bnum
\item[i)] $\nu_{\sigma^{-1}}^0:=\nu_{\sigma^{-1}}(v_0,u_0,.)$ the isometry of $N$ defined by $\sigma^{-1}$ at $(v_0,u_0)\in \Omega$,
\item[ii)]
$P^h_\sigma$ denotes the parallel transport with respect to $h$
along some curve $\gamma: [0,1]\to N$ with $\gamma(0)=p$ and
$\gamma(1)=\nu_{\sigma^{-1}}(v_0,u_0,p)$. \enum
 In particular,
   the full holonomy group of the Lorentzian manifold $M$  is given as
\begin{equation}\label{methodhol} \Hol_{\pi(\wt{q})}(M)=
\left\{\left.
 \hat{\phi}(\s) \right| \sigma\in \Gamma \right\}\cdot
\Hol_{p}(N,h)\ltimes \rrn.
\end{equation}
\etheo \bprf The construction of $g^{f,h}$ and the conditions on $f$
imply by Proposition \ref{satz-realisation-standard} that the
holonomy group of $(\tem, g^{f,h})$ is given by the group
$\Hol_p(N,h)\ltimes\rrn$.
 Furthermore, the vector field $\del_v$ is parallel and so is its push forward $\sigma_*\del_v$ by an isometry $\sigma$.
 Indeed, for an isometry  $\sigma$ we have
 \[0\ =\ \sigma_*(\tnab_X\del_v)\ =\ \tnab_{(\sigma_* X)}(\sigma_*\del_v)\]
 for all vector fields $X$ on $\tem$. By the conditions on $f$,  $\tem$ is indecomposable,
 which implies that $\sigma_*\del_v$ is a constant multiple of $\del_v$, as otherwise $\del_v$ and
 $\sigma_*\del_v$ would span a non-degenerate parallel distribution. Hence, we write the isometry $\sigma$
 in components according to $\tem=\Omega\times N$ and using global  coordinates $(v,u)$ on $\Omega$ as
 $\sigma=(v\circ \sigma,u\circ \sigma, \nu_\sigma)$. Then, when $\del_v(v\circ \sigma)$ denotes the
 directional derivative $\del_v(v\circ \sigma)(p)=d(v\circ\sigma)_p(\del_v(p))$,  we have
 \[ \sigma_*\del_v=\del_v(v\circ \sigma)\circ \sigma^{-1} \cdot \del_v\] and that  $\del_v(v \circ \sigma)$ is
 constant, i.e., there is a constant $a_\sigma\in \rr^*$ and  a function $\tau_\sigma$ of $u$ and $p\in N$
 such that $v\circ \sigma (v,u,p)=a_\sigma v+\tau_\sigma(u,p)$.
Also, because of
\[\d\sigma(\del_v)=
a_\sigma \del_v+d(u\circ\sigma)( \del_v) \del_u+d\nu_\sigma(\del_v )\]
we get  that
 $u\circ \sigma$ and $\nu_\sigma$ do not depend on $v$.
 We also note that $\sigma_*X$ is still orthogonal to $\del_v$ for $X\in \Gamma(TN)$. This implies that
 $d\sigma(X)\in \rr\cdot \del_v\+TN$ and thus
 that $d(u\circ \sigma)(X)=0$ which
 shows that $u\circ \sigma$ is a function only of the coordinate $u$. We also note that
 \[1=g(\del_v,\del_u)=\sigma^*g(\del_v,\del_u)=a_\sigma \frac{d}{du}(u\circ\sigma),\]
 which shows that $u\circ \sigma=\frac{1}{a_\sigma}u+b_\sigma$ with a constant $b_\sigma\in \rr$. Finally
 \[
 h(X,Y)=g(X,Y)=\sigma^*g(X,Y)=\nu_\sigma^*h(X,Y)\]
 for all $X,Y\in TN$ shows that $\nu_\sigma(v,u,.)$ is an isometry of $(N,h)$
This  proves formula~\eqref{isometry}.

In order to prove the result about the holonomy group,  consider a
curve  \[ \wt\delta(t)=( v(t):=v\circ \wt\delta(t),
u(t):=u\circ\wt\delta(t),\gamma(t))\] in $\tem$ with
$\wt{\delta}(0)=\wt q=(v_0,u_0,p)$ and
$\wt\delta(1)=\sigma^{-1}(\wt{q})$,  $\gamma$ a curve in $N$ with
$\gamma(0)=p$ and $\gamma(1)=\nu_{\sigma^{-1}}(\wt q)$.
Then the formulae in the proof of of
Proposition~\ref{satz-realisation-standard} imply that the parallel
transport along $\wt\delta(t)=( v(t), u(t),\gamma(t))$  is given in
the decomposition
\[
\rr\cdot\del_v\+TN\+\rr\cdot (\del_u-f\del_v)\] as
\[
\wt P_{\wt \delta} =  \begin{pmatrix}1& *&*\\
 0& P^h_{\gamma}&*\\
 0&0&1 \end{pmatrix}.\]
 Secondly, we have seen that the differential of $\sigma$ at $\sigma^{-1}(\wt q)$ is given as
 \[
 \d\sigma_{\sigma^{-1}(\wt q)}=
 \begin{pmatrix}
a_\sigma &*&* \\ 0& \d^N\nu_\sigma|_{\sigma^{-1}(\wt q)} & * \\
0&0&a_\sigma^{-1}\end{pmatrix} =  \begin{pmatrix}
a_\sigma &*&* \\ 0& (\d\nu^0_{\sigma^{-1}}|_p)^{-1} & * \\
0&0&a_\sigma^{-1}\end{pmatrix}.
 \]

Having this, we can apply formula~(\ref{parallel-quotient}) in
Proposition~\ref{coversatz} directly. For a loop $\delta:[0,1]\to
\tem/\Gamma$ at $\pi(\wt q)$, the lift
 is
 a curve $\wt\delta$ in $\tem $ such that $\wt{\delta}(0)=\wt q$ and $\wt{\delta }(1)=\sigma^{-1}(\wt q)$ for a $\sigma \in \Gamma$.
  Hence, by formula~(\ref{parallel-quotient}) in Proposition~\ref{coversatz} the parallel transport along $\delta$ is given by
 \[ P_\delta=\begin{pmatrix}
a_\sigma &*&* \\ 0&(\d \nu^0_{\sigma^{-1}}|_p)^{-1} \circ P^h_{\gamma}&*\\
0&0&a^{-1}_\sigma
\end{pmatrix}.
 \]
The same argument as in the proof of Proposition \ref{Satz-full-hol}
shows, that we can use the matrix
\[ \hat{\phi}(\sigma) = \begin{pmatrix}
a_\sigma &0&0 \\ 0&(\d \nu^0_{\sigma^{-1}}|_p)^{-1} \circ P^h_{\gamma}& 0\\
0&0&a^{-1}_\sigma
\end{pmatrix}\] as representative of the class $\Phi(\sigma)\in Hol_q(M)/Hol_{\wt q}(\wt M) $.
This proves the second statement. \eprf

The examples in the next section will be based on this Theorem. For
most of them we will apply the following
\bfolg\label{method-folge} Let $(N,h)$ be an $n$-dimensional
Riemannian manifold, $\Gamma$ a properly discontinuous group of
isometries of $(N,h)$, and $f$ a $\Gamma$-invariant function on $N$
satisfying property~\eqref{hess}. Fix a not necessarily finite  set
of generators $(\gamma_1, \gamma_2, \ldots  )$ of $\Gamma$.
Corresponding to these generators of $\Gamma$ fix a sequence of
integers $\underline{m}=(m_1, m_2, \ldots )$ and of  real numbers
$\underline{\lambda}:=(\lambda_1, \lambda_2, \ldots )$.
 \bnum
 \item
 Consider the Lorentzian metric
 \[g^{f,h}=2dvdu+2fdu^2 +h\]
 on $\tem:=\rr^2 \times N$. Suppose that the
 group $\Gamma_{\underline{m}}$ generated in the isometry group of $(\tem,g^{f,h})$ by the isometries
 \[
 \sigma_{m_i}(v,u,p):=((-1)^{m_i}v,(-1)^{m_i} u, \gamma_i(p))
 \]
 for $i=1, 2, \ldots $ is of quotient type or restrict to $\tem_0:= (\rr^2\setminus \{(0,0)\}) \times N$ otherwise. Then
 $\Gamma_{\underline{m}}$ is properly discontinuous and the holonomy group of the Lorentzian
 manifold $(M:=\tem_{(0)}/\Gamma_{\underline{m}},g^{f,h})$ at $\pi(v,u,p)$
 is given as
 \[
{L\Gamma_{\underline{m}}}\cdot\left( \Hol_p(N)\ltimes \rrn\right),\]
 where $ L{\Gamma_{\underline{m}}}$ is the group that is
  generated in $\O(1,n+1)$ by
 the linear maps \[
 \begin{pmatrix} (-1)^{m_i}& 0&0
 \\
 0&\phi(\gamma_i)&0
 \\
 0&0&(-1)^{m_i}
 \end{pmatrix}\ \in\  \Z_2\times \phi(\Gamma),
 \]
 where $\phi(\gamma_i)$ is the representative for the class $\Phi(\gamma_i)\in \Hol_{\pi(p)} (N/\Gamma)/\Hol_p(N)$ as described in
 of Proposition~\ref{coversatz}.
 In particular, for appropriate choices of $\underline{m}$, $(M,g^{f,h})$ is not time-orientable,
 admits a parallel null line bundle but no parallel vector field.
 \item Set $\Omega:=\{(v,u)\in\rr^2\mid u>0\}$ and consider the Lorentzian metric
 \[g^{f/u^2,h}=2dvdu+\frac{2}{u^2}fdu^2 +h\]
 on $\tem:=\Omega\times N$.
 Suppose that the group $\Gamma_{\underline{\lambda}}$ generated in the isometry group of $(M,g^{f/u^2,h})$ by the isometries
 \[
 \sigma_{\lambda_i}(v,u,p):=(\e^{\lambda_i}v,\e^{-\lambda_i} u, \gamma_i(p))
 \]
 for $i=1, 2, \ldots $ is of quotient type. Then $\Gamma_{\underline{\lambda}}$ is properly discontinuous and
 the holonomy group of the Lorentzian manifold
 $(M:=\tem/\Gamma_{\underline{\lambda}},g^{f/u^2,h})$ at $\pi(v,u,p)$
 is given as
 \[
{L\Gamma_{\underline{\lambda}}}\cdot (\Hol_p(N)\ltimes \rrn),\]
 where $ L{\Gamma_{\underline{\lambda}}}$ is the group that is
  generated in $\O(1,n+1)$ by
 the linear maps
 \[
 \begin{pmatrix} \e^{\lambda_i}& 0&0
 \\
 0&\phi( \gamma_i)&0
 \\
 0&0&\e^{-\lambda_i}
 \end{pmatrix}\ \in \ \R^+\times \phi(\Gamma).
 \]
 In particular, $(M,g^{f,h})$ is time-orientable, admits a parallel null line bundle, but,
 for appropriate choices of $\underline{\lambda}$, admits no parallel vector field.
 \enum
\efolg \bprf Since $\Gamma_{\underline{m}}$ and
$\Gamma_{\underline{\lambda}}$ are of quotient type, by
Lemma~\ref{properlemma} they are properly discontinuous as well. By
the constructions of the Lorentzian metrics, they also act
isometrically. Then the formula for the holonomy group of the
Lorentzian quotient follows from formula~\eqref{methodhol} in
Theorem~\ref{method}. \eprf

\bbem\label{compact-remark} Obviously, the constructions presented
in this section always give non-compact examples. But in simple
situations they can be modified in order to obtain compact
Lorentzian manifolds by replacing $\Omega$ by a torus $S^1\times
S^1$ with a suitable metric. Start with a compact Riemannian
manifold $(N,h)$ with holonomy $G$ and a function $f$ on $S^1\times
S^1\times N$, even in the first two entries and satisfying condition
\eqref{hess}. Then the Lorentzian metric on $M=S^1\times S^1\times
N$ given by
\[
g^{f,h}=2d\vf d\theta+f d\theta^2+h,\] where $\vf$ and $\theta$ are
standard angle coordinates on $S^1\times S^1$, has holonomy
$G\ltimes \rrn$ or $(\rr^+\ltimes G)\ltimes \rrn$. Then we can
consider the same involution on $M$ as in Proposition
\ref{satz-realisation-standard}, in coordinates
\[(\vf, \theta, p)\mapsto (\vf + \pi, \theta + \pi, p), \]
to obtain compact Lorentzian manifolds with holonomy $(\Z_2\times
G)\ltimes \rrn$ or $(\rr^*\times G)\ltimes \rrn$.

This can be generalised to the situation when we have a cocompact
properly discontinuous group $\Gamma^N$ of isometries of $(N,h)$
being generated by $\gamma_1,\gamma_2 , \ldots $. We choose the
function $f$ to be $\Gamma^N$-invariant and independent of $\theta$
and $\vf$, fix natural numbers $\underline{m}=(m_1,m_2, \ldots)$ and
consider the group $\Gamma_{\underline{m}}$  of isometries of
$(M,g^{f,h})$ which is generated by
\begin{equation}\label{compact-coupled}
(\vf,\theta, p)\mapsto ( \vf + m_i \pi, \theta + m_i\pi,
\gamma_i(p)),\end{equation} for $i=1, 2, \ldots$. Then
$\Gamma_{\underline{m}}$ is of quotient type and the holonomy of
$M/\Gamma_{\underline{m}}$ is contained in the group
$(\Z_2\times\Hol(N/\Gamma^N))\ltimes\rrn$ but might have a coupling
between the $\Z_2$ and the $\Hol(N/\Gamma^N)$ part. \ebem

\section{Examples with disconnected holonomy}
\label{section-pp-waves}
\label{examples} In the following we will consider quotients of
certain Lorentzian manifolds by a discrete group of isometries. These examples
will illustrate some of the characteristic features of the holonomy groups that
can be obtained by the construction given in Theorem~\ref{method}.

\subsection{Flat manifolds}
We will start off with the flat case and
give an example of a flat Lorentzian manifold with indecomposable,
non irreducible full holonomy (see Remark \ref{bem1}).
Let $\R^{1,n+1}$ be the Minkowski space and $E(1,n+1)=\O(1,n+1)
\ltimes \R^{1,n+1}$ its isometry group. Any isometry $\gamma$ has
the form  $\gamma(x) = A_{\gamma}x + v_{\gamma}$, where $A_{\gamma}
\in \O(1,n+1)$ and $v_{\gamma}\in \R^{n+2}$. For a discrete subgroup
$\Gamma \subset E(1,n+1)$ we denote by
\[ L_{\Gamma}:= \{ A_{\gamma} \mid \gamma \in \Gamma \}
\subset \O(1,n+1)\] its linear part. Then, by Proposition
\ref{coversatz}, the full holonomy group of a flat space-time
$\R^{1,n+1}/\Gamma$ is given by
\[ \Hol(\R^{1,n+1}/\Gamma) = L_{\Gamma}.\]
In many cases, the holonomy group of $\R^{1,n+1}/\Gamma$ stays
trivial or acts decomposable. For example, if $\R^{1,n+1}/\Gamma$ is
a complete homogeneous flat Lorentzian space, then $\Gamma$ is a
group of translations, hence its holonomy group is trivial (see
\cite{wolf}). The same is true for non-complete flat homogeneous
Lorentzian spaces (see \cite{Duncan-Ihrig} for the result). Looking
at the affine classification of compact 3-dimensional flat
space-times which are proper in the sense that their holonomy is
contained in $\SO^0(1,2)$ in \cite[Sec. 3.6]{wolf},
one finds all with
trivial, with decomposable as well as with indecomposable holonomy
group. Typically, an indecomposable holonomy group appears, if
$\Gamma$ contains an element $\gamma$, such that $A_{\gamma}$ has a
lightlike eigenvector. We show this with a 4-dimensional example.

For a fixed $\theta \in \rr$ we consider the matrix in $A_{\theta}
\in \SO(1,3)$:
\[
A_\theta:=\begin{pmatrix} 1 & -\cos \theta& \sin\theta & -\einhalb
\\
0 & \cos \theta&-\sin\theta & 1\\
0 & \sin \theta&\cos \theta & 0\\
0&0&0&1
\end{pmatrix}.
\]
where we used the basis $(\ell,e_1,e_2,\ell^*)$ as in the first
section. Then, by $\,k\mapsto (A_\theta)^k\,$ we obtain an immersion
of the finite group $Z_m$ into $\O(1,3)$ if $\theta $ is a rational
multiple of $\pi$, and an immersion of the group $\Z$ into $\O(1,3)$
if $\theta$ is an irrational multiple of $\pi$. If we denote the
image of this immersion by $H$, the connected component of $H$ is
given by the identity and thus acts decomposable in a trivial way,
but the group $H$ does not admit a non-degenerate invariant
subspace, and hence it acts indecomposable. The group $H$ can be
easily realized as the holonomy group of a flat space-time. For that
we consider coordinates $(v,u,x,y)$ on $\rr^{4}$ and the flat
Minkowski metric
\[g_0 =2\d u\d v + \d x^2+\d y^2. \]
For $\theta\in \rr$, let $\Gamma \subset E(1,3)$ be the discrete
group of isometries generated by
\[ \phi_\theta (v,u,x,y)=
\  A_\theta (v,u,x,y) +
(0,1,0,0).\] $\Gamma$ acts freely and properly. Hence, the quotient
$\rr^{1,3}/\Gamma$ is a flat Lorentzian manifold with trivial
restricted holonomy and
full holonomy $H$.

\subsection{pp-waves} The next examples will be quotients of
Lorentzian manifolds which are usually referred to as {\em pp-waves}
(see Remark \ref{ppbem}). We will define them as follows: Let $\tem$
be an open set in $\rr^{n+2}$ on which we fix global coordinates as
$(v,u,x^1,\ldots , x^{n})$. Let $f$ be a smooth function on $\tem$
that does not depend on the $v$ coordinate, i.e. $\del_v f=0$. A
pp-wave metric on $\tem$ is a Lorentzian metric $g^{f}$ which, in
these coordinates, is given as
\begin{eqnarray}\label{ppmetric}
g^{f}&=& 2dvdu \,+\, 2f du^2 \,+\, \sum_{i=1}^{n}(\der x^i)^2.
\label{pp3}
\end{eqnarray}
The vector fields
\begin{eqnarray}\label{frame}
\del_v, &  \del_i,&  \del_u - f\del_v\end{eqnarray} form a global
frame in which the metric $g^{f}$ is given as
\[g^{f}=\begin{pmatrix}0&0&1\\0&\1_n&0\\1&0&0\end{pmatrix}.
\]
Here we used the obvious notation $\del_v:=\frac{\partial}{\partial
v}$, $\del_u:=\frac{\partial}{\partial u}$ and
$\del_i:=\frac{\partial} {\partial x^i}$, $i=1,\ldots , n$. Here and
in the following, all matrices  are written with respect to the
basis of the tangent spaces given by these vector fields at the
corresponding points.

\begin{satz}\label{satz-hol-pp}
The holonomy group of an (n+2)-dimensional pp-wave is abelian and if
the matrix $\big(\del_i\del_j f\big)$ is non-degenerate at a point,
the holonomy group is given by
\[\R^{n} \simeq \left\{\begin{pmatrix} 1 & y^t& \tfrac{1}{2}y^t y \\
0&\1_n & -y
\\
0&0&1
\end{pmatrix} \;\; \Big| \;\;  y\in \rr^{n}\right\}.\]
Conversely, any Lorentzian manifold with this holonomy group is
locally a pp-wave.
\end{satz}
\begin{proof} pp-waves
are a special case of the manifolds considered in
Proposition \ref{satz-realisation-standard}. Hence the first
statement follows from the proof of Proposition
\ref{satz-realisation-standard}. For the converse statement see
\cite{leistner01}. \end{proof}
The formula (\ref{parallel-general-wave}) in the proof of
Proposition \ref{satz-realisation-standard} shows in addition that
for a curve $\delta(t)=(v(t),u(t),\gamma(t))$ in a pp-wave, the parallel
transport along $\delta$ is given by matrices of the form
\begin{eqnarray} P_{\delta} = \left(%
\begin{array}{ccc}
  1 & * & * \\
  0 & \1_n & * \\
  0 & 0 & 1 \\
\end{array}%
\right).   \label{transport-pp} \end{eqnarray}

\subsection{Lorentzian symmetric spaces}
As a first class of manifolds that are covered by a pp-wave, we consider Lorentzian symmetric
spaces. Let $(M,g)$ be an indecomposable Lorentzian symmetric space
of dimension $n+2 \geq 3$. Then its transvection group $G(M)$ is
either solvable or semi-simple \cite{Cahen-Wallach:70}. In the
latter case, $(M,g)$ is a space of constant sectional curvature
$\kappa \not =0$, hence its holonomy group acts irreducibly. The
case of solvable transvection group was described by Cahen and
Wallach in \cite{Cahen-Wallach:70,Cahen-Parker:80}, see also
\cite{Neukirchner:03}. The simply-connected models are given by the
following special pp-waves: Let $\lu =( \lambda_1, \ldots,
\lambda_{n})$ be an $n$-tupel of real numbers $\lambda_j \in {\R}
\backslash \{0\}$. Then the pp-waves $M_{\lu}:= ( {\R}^{n+2},
g_{\lu} )$, where
\[
 g_{\lu}:= 2\d v\,\d u +
\sum\limits^{n}_{j=1} \lambda_j (x^j)^2 \,\d u^2 + \sum\limits^{n}_{j=1}
(\d x^j)^2,
 \]
are symmetric. If $\lu_{\pi} =( \lambda_{\pi(1)}, \ldots ,
\lambda_{\pi (n)} )$ is a permutation of $\lu$ and $c >0$, then
$M_{\lu}$ is isometric to $M_{c\lu_{\pi}}$. By Proposition
\ref{satz-hol-pp} the holonomy group of $M_{\lu}$ is abelian and
isomorphic to $\R^n$.

Any indecomposable solvable Lorentzian symmetric space $(M,g)$ of
dimension $n+2\geq 3$ is isometric to a quotient $\,M_{\lu}/\Gamma$,
where $\lu \in ({\R} \backslash \{ 0 \})^{n}$ and $\Gamma$ is a
discrete subgroup of the centralizer $Z_{I(M_{\lu})} (G(M_{\lu}))$
of the transvection group $G(M_{\lu})$ in the isometry group
$I(M_{\lu})$ of $M_{\lu}$. For the centralizer
$\,Z_{\lu}:=Z_{I(M_{\lu})} (G(M_{\lu}))\,$ the following is known
(see \cite{Cahen-Kerbrat:78}):
\begin{enumerate}
\item[1)]
If there is a positive $\lambda_i$ or if there are two numbers
$\lambda_i, \lambda_j$ such that $\frac{\lambda_i}{\lambda_j}
\not\in \{ q^2 \mid q\in \QQ\}$, then $\,Z_{\lu}= \{ t_{\alpha} \mid
\alpha \in \R\}$, where $t_{\alpha}$ is the translation
$t_{\alpha}(v,x,u)=(v+ \alpha,u, x)$.
\item[2)]
If $\lambda_i =- k^2_i$ and $\,\frac{k_i}{k_j} \in {\QQ}\,$ for all
$i,j \in \{ 1, \ldots , n\}$, then $Z_{\lu}$ is generated by $\{
t_{\alpha} \mid \alpha \in \R\}$ and the isometry
\[
\varphi_{\beta}(v,u,x) =(v,  u + \beta\pi, (-1)^{\beta k_1} x^1, \ldots ,
(-1)^{\beta k_{n}} x^{n} ),
 \]
where $\beta:= \min \{ r \in \R^+ \, \mid \, r k_i \in \Z \; \mbox{
for all} \;  i\in \{1,\ldots,n\}\, \}$.
\end{enumerate}

Clearly, any discrete subgroup in $Z_{\underline{\lambda}}$ is
generated by a translation $t_\alpha$ in the $v$-component and a
power of $\vf_\beta$. Hence, denote by $\Gamma_{m,\alpha}$ the
discrete group of isometries generated by a translation $t_{\alpha}$
and by the isometry $\varphi_{\beta}^m$ with $\alpha \in \R$ and
$m\in \N_0$ ($m=0$ in the first case). $\Gamma_{m,\alpha}$ is
obviously properly discontinuous.
 Consequently, any indecomposable solvable Lorentzian
symmetric space $M$ is isometric to $M_{\lu}/\Gamma_{m,\alpha}$ for
an appropriate $m$ and $\alpha$.

\begin{satz} \label{satz-fullhol-symm}
Let $M_{\lu}/\Gamma_{m,\alpha}$ be an indecomposable solvable
Lorentzian symmetric space. Then the full holonomy group is given by
\[ \Hol(M_{\lu}/\Gamma_{m,\alpha})
 \simeq \left\{ \begin{array}{ll} \R^n & \;\; \;\;  \mbox{if $m$ even,}
 \\
                             \Z_2 \ltimes \R^n = \{Id, S_m\} \ltimes \R^n
                             & \;\; \;\;  \mbox{if $m$ odd,}
 \end{array}\right.\]
where in case of $m$ odd $S_m$ is the reflection $\,S_m =
diag((-1)^{m\beta k_1},\ldots,(-1)^{m\beta k_n}) \in\O(n)$. In
particular, any indecomposable solvable Lorentzian symmetric space
is time-orientable and admits a global parallel lightlike vector
field.
\end{satz}
\begin{proof}
In order to determine the holonomy group of $M$
we apply formula~(\ref{methodhol}) of Theorem \ref{method}. As the
parallel transport in $M_{\lu}$ is of the form \eqref{transport-pp},
we only have to compute the differentials of the isometries in
$\Gamma_{m,\alpha}$. But these differentials are clearly given by
the identity or by
\[ d\vf_\beta^m=
\left(
\begin{array}{ccc}
  1 & * & * \\
  0 & S_m & * \\
  0 & 0 & 1 \\
\end{array}
\right).\] This proves the Proposition.
\end{proof}

\subsection{Quotients of pp-waves}
Now, by applying Theorem~\ref{method} and Corollary~\ref{method-folge} we construct two types of quotients of pp-waves for which the
full holonomy group has also elements in the dilatation part $\R^*$
of $\mbox{Stab}_{\O(1,n+1)}(L)$.
We fix linear independent vectors $(a_1,\ldots,a_p)$ in $\R^n$ and
consider the discrete group of translations of $\R^n$, \[\Gamma_{(a_1,\ldots,a_p)}:=\{ t_a  \mid
 a =  \sum_{i=1}^p m_ia_i, \,  m_i\in \Z \},\]
 which is generated by the translations by the $a_i$.
Furthermore, we
fix a $\Gamma_{(a_1,\ldots,a_p)}$-invariant function $f\in
C^{\infty}(\R^n)$ satsifying the property \eqref{hess}.

In the first example we consider the metric
\[ \tg^{f} = 2 dvdu + 2f(x)du^2 + \sums_{i=1}^n (dx^i)^2 \]
on $\tM=\rr^2\times \rr^n=\R^{n+2}$. Then, for $ t_a \in \Gamma$
with $a=\sum_{i=1}^p m_ia_i$ we define a isometry $\varphi_{a}$ of
$(\tem,g^f)$  by
\[ \varphi_{a}(v,u,x) := ((-1)^{\sum_i m_i}v,
(-1)^{\sum_i m_i}u , x+a). \] and apply Theorem \ref{method} to the
properly discontinuous group \[\Gamma:=\{ \varphi_{a} \mid t_a\in
\Gamma_{(a_1,\ldots,a_p)} \} \simeq \Z^p. \]

    Then, by Corollary ~\ref{method-folge}, the full holonomy group of the Lorentzian manifold $M= \R^{n+2}/\Gamma$ with the
    Lorentzian metric induced by $\wt g^{f}$ is given by
    \begin{equation*}
        \left\{  \begin{pmatrix}\pm 1 & y^t &\ast \\ 0& \1_n &\ast \\0&0&\pm 1 \end{pmatrix} \;\; \Big| \;\; y\in \rr^n\right\}
         \simeq \Z_2\ltimes \rr^n,
    \end{equation*}
    whereas the restricted holonomy is given by $\rr^n$. In particular,
    $M$ is not time-oriented and admits a parallel lightlike line, but no parallel
    lightlike vector field.

In order to construct an example which is time-orientable we involve
Lorentz boosts on $\rr^{1,1}$ in this construction and consider now
the pp-wave metric defined by $\frac{1}{u^2}{f}(x)$ on the
half-space $\tem:=\Omega\times\rrn$ with $\Omega=\{(v,u)\in
\rr^{2}\mid u>0\} \subset \R^{n+2}$, i.e.,
\begin{equation}\label{3zmetric}
    \tilde{g}^{f/u^2} = 2d v d u + \frac{2}{u^2}{f}(x)d u^2 + (d x^{1})^2 + \ldots + (\der x^{n})^{2}.
\end{equation}
Let $\lambda_{1},\ldots,\lambda_{p} \in \R$ be linearly independent
over $\mathbb{Q}$ and define for $t_a\in \Gamma$, $a= \sum_im_ia_i$,
the isometries
\begin{equation*}
    \psi_{a}(v,u,x) := (\e^{\sum_i m_i\lambda_{i}}v  ,\e^{-\sum_i m_i\lambda_{i}}u, x+a).
\end{equation*}
on $(\tM,\tg^{f/u^2})$. The group $\Gamma := \{ \psi_a \mid t_a\in
\Gamma_{(a_1,\ldots,a_p)} \}\simeq \Z^p$  is of quotient type, hence
properly discontinuous, and $\tilde{g}^{f/u^2}$ induces a Lorentzian
metric on the quotient $M:=\tem/\Gamma$.
Since $\lambda_1,\ldots,\lambda_p$ are independent over $\QQ$, the
epimorphism  $F: \pi_{1}(M) \twoheadrightarrow \Hol_p({\cal L},
\nabla^{\cL})$ is injective. Hence, the full holonomy group of the
Lorentzian manifold
    $M=\tem/\Gamma$
     is given by
    \begin{equation*}
         \left\{ \begin{pmatrix}
                                \e^{\sum_{j}{k_{j}\lambda_{j}}} & y^t &\ast\\
                                0& 1 &\ast \\
                                0&0&\e^{-\sum_{j}{k_{j}\lambda_{j}}}
                        \end{pmatrix} \;\; \Big| \;\;  y \in \rr^{n}, k_{j} \in \Z\right\}\simeq \Z^{p} \ltimes \rr^{n},
    \end{equation*}
    whereas the restricted holonomy group is given by $\rr^{n}$. In particular, the manifold $M$
    is time-orientable admitting a parallel null line, but no
    parallel lightlike vector field.

\subsection{Coupled holonomy}
The aim of this section is to modify the examples of the previous section in a way so that the
holonomy group features a coupling between the linear part in $\O(n)$ and the scaling component
in $\rr^*$ in the sense that some group elements act simultaneously on $\rr$ and on $\rr^n$.
The first examples are direct consequences of Corollary~\ref{method-folge}.

First we produce 4-dimensional examples. The first class of examples
is based on 2-dimensional flat Riemannian space-forms which are
either diffeomorphic to a M\"obius strip or to a Kleinian bottle
(for both cases, see \cite[Chap. 2.2.5]{wolf}). In case of the
M\"obius strip, this manifold is given by $\R^2/\Gamma^1$, where
$\Gamma^1$ is infinite cyclic generated by $\gamma=(B,t_b) \in
\O(2)\ltimes \R^2$, with a reflection  $B$ fixing the non-zero
vector $b\in \rr^2$. In case of the Kleinian bottle this manifold is
given by $\R^2/\Gamma^2$, where $\Gamma^2$ is generated by
$\gamma_1=(B,t_b)$ and $\gamma_2=(I,t_a)$ in $\O(2)\ltimes \R^2$,
and $B$ is again a reflection with $B(b)=b\not = 0$ and
$B(a)=-a\not=0$. For the construction of pp-waves,  we fix two
$\Gamma^i$-invariant functions $f_i$ on $\R^2$ with a non-degenerate
Hessian.

Consider the pp-wave metrics
\[ \tg^{f_i} := 2dvdu + {2} f_i(x^1,x^2) du^2 + (dx^1)^2 + (dx^2)^2\]
on $\tem:=\rr^4$
and
\[ \tg^{f_i/u^2} := 2dvdu + \frac{2}{u^2} f_i(x^1,x^2) du^2 + (dx^1)^2 + (dx^2)^2\]
on $\tM:= \{(v,u,x^1,x^2) \mid u>0 \}$, both
 with holonomy $\rr^2$. We fix numbers  $m,m_1,m_2\in \N$ and $ \lambda,\lambda_1,\lambda_2$ in $\rr$ and define
 the groups of isometries $\Gamma^1_m$, $\Gamma^1_\lambda$, $\Gamma^2_{m_1,m_2} $ and $\Gamma^2_{\lambda_1,\lambda_2}$
 as in Corollary~\ref{method-folge}, with respect to these numbers and the fixed generators $\gamma$ of $\Gamma^1$
 and $\gamma_1$ and $\gamma_2$ of $\Gamma^2$. These groups are of
 quotient type.
 Then, by Corollary~\ref{method-folge} we obtain the following holonomy groups for the quotients
\be
 \Hol(\tM/\Gamma^1_{m},g^{f^1})& = &\left \{ \left(%
\begin{array}{ccc}
  (-1)^{km}  & w^t & * \\
  0 & B^k & * \\
  0 & 0 & (-1)^{km} \\
\end{array}%
\right) \;  \Big| \; w\in \rr^2, k\in \Z \right\},
\\
 \Hol(\tM/\Gamma^2_{m_1m_2},g^{f^2}) &= &
 \left \{ \left(%
\begin{array}{ccc}
  (-1)^{k_1m_1+k_2m_2}  & w^t & * \\
  0 & B^{k_1} & * \\
  0 & 0 &   (-1)^{k_1m_1+k_2m_2} \\
\end{array}%
\right)  \; \Big| \;w\in \rr^2, k_1,k_2\in \Z \right\},
\\
\Hol(\tM/\Gamma^1_{\lambda},g^{f^1/u^2}) &=& \left \{ \left(%
\begin{array}{ccc}
\e^{k\lambda}  & w^t & * \\
  0 & B^k & * \\
  0 & 0 & \e^{-k\lambda} \\
\end{array}%
\right) \; \Big| \;w\in \rr^2, k\in \Z \right\}, \\
 \Hol(\tM/\Gamma^2_{\lambda_1\lambda_2},g^{f^2/u^2}) &=& \left \{ \left(%
\begin{array}{ccc}
\e^{k_1\lambda_1+k_2\lambda_2}  & w^t & * \\
  0 & B^{k_1} & * \\
  0 & 0 &   \e^{-k_1\lambda_1-k_2\lambda_2} \\
\end{array}%
\right) \; \Big| \; w\in \rr^2, k_1,k_2\in \Z \right\}.
\ee
Thus, by
different choices of the numbers $m,m_1,m_2,
$ and  $
\lambda,\lambda_1,\lambda_2$
 we can realize the groups $\Z_2\ltimes
\R^2$,  $(\Z_2\oplus \Z_2)\ltimes \R^2$, $\Z\ltimes \R^2$, and
$(\Z\oplus\Z) \ltimes \R^2$, where the integer parts are immersed
into $(\R^*\times \O(2))$ coupling both factors, as holonomy groups
of 4-dimensional not time-orientable Lorentzian manifolds.

Now we start the construction with a group of isometries of $\R^2$
that is not properly discontinuous.
  Fix an
angle $\theta$ and consider the group $\Gamma_\theta$ generated by
the rotation $D_\theta$ by $\theta$. This group is isomorphic to
$\Z$ or $\Z_p$ depending on whether $\theta$ is a rational multiple
of $\pi$ or not. We fix a rotational invariant function $f$ on
$\rr^2$ with property \eqref{hess} and consider again the pp-wave
metric
\[g^{f/u^2}=2dvdu +\frac{2}{u^2}f+(dx^1)^2+(dx^2)^2
\]
on $\tem= \Omega\times \rr^2$ with $\Omega=\{(v,u)\in \rr^2\mid
u>0\}$. Then, for $\lambda, c \in \rr$,  the group $\Gamma$
generated by
\[\varphi (v,u,x):=(\e^{\lambda}v+c, \e^{-\lambda}u, D_\theta(x))
\] acts on $\tem$ by isometries of $g^{f/u^2}$ and is properly discontinuous. Indeed,
since we restrict the action to $\{u>0\}$ the group of
diffeomorphisms on $\Omega$ generated by
\[(v,u)\mapsto ( \e^{\lambda}v+c, \e^{-\lambda}u)
\] is of quotient type and therefore properly discontinuous. Using this, we obtain that $\Gamma$
is of quotient type as well and Lemma 2 applies again. Hence,
$\tem/\Gamma$ is a Lorentzian manifold with holonomy $\Z\ltimes
\rrn$ if $\lambda\not=0$, i.e.,
\[
\Hol(\tem/\Gamma,g^{f/u^2})
=
 \left \{ \left(%
\begin{array}{ccc}
\e^{k\lambda}  & w^t & * \\
  0 & D_\theta^k & * \\
  0 & 0 & \e^{-k\lambda} \\
\end{array}%
\right) \; \; \Big| \;\; w\in \rr^2, k\in \Z \right\}. \] In a
similar way we can construct non-time-orientable Lorentzian
manifolds by taking $\Omega=\rr^2\setminus \{(0,0)\}$ and
$\varphi(v,u,x):=(-v,-u,D_\theta)$. Here the holonomy $\Z\ltimes
\rrn$ or $\Z_p\ltimes \rrn$ acts as $\Z_2$ in the $\rr^*$ part.

\bbem Note that, if in the prevous example we assume $\theta \not = \pi$ , the orthogonal part
$\pr_{\O(2)}(\Hol(\tem/\Gamma,g^{f/u^2}))\simeq \Gamma_{\theta}$
cannot be realised as holonomy group of a complete Riemannian
manifold.
If there was a complete
2-dimensional Riemannian manifold $M^2$ with holonomy
$\Gamma_\theta$, then $M^2$ would be a flat Euclidean space-form.
But the holonomy groups of flat 2-dimensional space-forms are known
to be trivial (for the plane, the cylinder and the torus) or
generated by a reflection (for the M\"obius strip and the Kleinian
bottle) (see \cite{wolf}, Chap. 2.2.5). We do not know if this is
also true if we drop the assumption of completeness.

This is in contrast to the situation when considering only the connected component of the holonomy group.
Here it was proven in \cite{leistnerjdg} that $\pr_{\SO(n)}(\Hol^0)$ is always the connected holonomy
of a Riemannian manifold, and, by results in Riemannian holonomy theory (an overview can be found in
\cite{joyce07}), can be realised as holonomy of a {\em complete} Riemannian manifold.
\ebem

In the same way, 5-dimensional examples can be constructed starting
with flat 3-dimensional Riemannian space-forms. These are classified
(see \cite[Chap. 3.3.5]{wolf}) and one can easily single out those
which are generated by discrete groups $\Gamma$ with non-trivial
linear part $L_{\Gamma}$. Then, similar constructions as above give
various examples where the $\R^*$ and the $\O(3)$ part of the
holonomy is coupled. We show this only for the 3-manifolds of type
$\cS^{\theta}_1$ in the list of 3-dimensional flat space-forms. Let
$\theta \in (0,2\pi)$ and let us consider the discrete group
$\Gamma\subset E(3)$ generated by $(A,t_a)\in \O(3)\ltimes \R^3$,
where $a\in \R^3$ is an eigenvector of $A$ and $A|_{a^\bot} =
D_{\theta}$ is the rotation by the angle $\theta$. We fix a function
${f} \in C^{\infty}(\R^3)$ which is $\Gamma$-invariant and has a
non-degenerate Hessian in a certain point and consider the pp-wave
metric $g^{f}$ and $g^{f/u^2}$ on the appropriate $\tem$, where we
choose the coordinates in such a way, that $a$ is a positive
multiple of $\del_1$. Let $m\in \N$ and $\lambda\in \R$ be fixed
numbers and denote by $\Gamma_{m}$ and $\Gamma_{\lambda}$  the
properly discontinuous group of isometries on $\tM$ as in
Corollary~\ref{method-folge}. Then the holonomy group of the
quotients is given by \be
 \Hol(\tM/\Gamma_{m},g^{f})& = &\left\{ \left(%
\begin{array}{cccc}
 (-1)^{km}& s & w^t & * \\
  0 & 1 & 0 & *\\
  0 & 0 & D_{k\theta} & * \\
  0 & 0 & 0 & (-1)^{km} \\
\end{array}%
\right) \; \Big| \; k\in \Z, s\in \rr, w\in \rr^2  \right\},
\\
 \Hol(\tM/\Gamma_{\lambda},g^{f/u^2})& = &\left\{ \left(%
\begin{array}{cccc}
 e^{k\lambda}& s & w^t & * \\
  0 & 1 & 0 & *\\
  0 & 0 & D_{k\theta} & * \\
  0 & 0 & 0 & e^{-k\lambda} \\
\end{array}%
\right) \; \Big| \; k\in \Z, s\in \rr, w\in \rr^2  \right\}. \ee
Again, by
appropriate choice of $\lambda$ and $\theta$ we can realise the
groups  $\Z \ltimes \R^3$ and $\Z_{2q} \ltimes \R^3$ as holonomy group of a 5-dimensional Lorentzian manifold,
where
$\Z$ and $\Z_{2q}$ are immersed into $\R^*\times \O(3)$ coupling both factors.

\subsection{An infinitely generated holonomy group}
Again by applying Corollary~\ref{method-folge} we will give an example of a 4-dimensional Lorentzian manifold
 with infinitely generated holonomy group arising as a quotient of a pp-wave.
Consider $\rr^2$ with the flat metric $h_0=(d x^1)^2+ (d x^2)^2$ and
restrict $h_0$ to the manifold $N:=\rr^2\setminus \Z^2$. The
fundamental group $\Gamma:=\pi_1N$ of $N$ is a free group infinitely
generated by the loops going around the holes,
 but the holonomy
group of $(N,h_0)$ is trivial.
 Fix a function $f$ on $N$ satisfying \eqref{hess}.  The universal cover of $N$ is $\rr^2$ and
  $N$ is the quotient of $\rr^2$ by $\pi_1N$. It is equipped with the $\Gamma$-invariant
   pull-back $h$ of the flat metric $h_0$ on $N$. Also the function $f$ pulls back to a $\Gamma$-invariant function on $\rr^2$.
We apply Corollary~\ref{method-folge} to $(\rr^2,h)$ and $\Gamma$.

Set $\Omega :=\{(v,u)\in \rr^2\mid u>0\}$ and define a Lorentzian metric on $\tem:=\Omega \times \rr^2$ by the usual procedure
\[\wt g^{f/u^2,h}=2dvdu+\frac{2}{u^2}f d u+ h.\]
Now we fix generators $( \gamma_1, \gamma_2, \ldots\  )$ of $\Gamma$
and real numbers $\underline{\lambda}:=( \lambda_1, \lambda_2,
\ldots \ )$ which are linearly independent over $\mathbb{Q}$ and
define the following isometries of $(\tem,g^{f/u^2,h})$:
\[\varphi_i(v,u,x)=(\e^{\lambda_i}v, \e^{-\lambda_i}u, \gamma_i(x)).\]
Let $\Gamma_{\underline{\lambda}}$ be the group of isometries of
$(\tem, \wt g^{f/u^2,h})$ that is generated by the $\varphi_i$ for
$i=1,2,\ldots\ $. Since the fundamental group $\Gamma=\pi_1N$ acts
as group of deck transformations of the universal covering $\pi:
\R^2 \to N$, it is properly discontinuous. Moreover, since $\Gamma$
is a free group, the sections
$(\Gamma_{\underline{\lambda}})_{\sigma} \subset
(\Gamma_{\underline{\lambda}})_{1}$ consist only of one element for
all $\sigma \in \Gamma$. Hence, $\Gamma_{\underline{\lambda}}$ is of
quotient type and by Lemma \ref{properlemma} properly discontinuous,
and the quotient $M=\tem/\Gamma_{\underline{\lambda}}$ is a smooth
Lorentzian manifold with the induced Lorentzian metric
$g^{f/u^2,h}$. Then Corollary~\ref{method-folge} shows that the
holonomy group of $(M,g^{f/u^2,h})$ is generated by the following
matrices
\[\begin{pmatrix}
\e^{\lambda_i}&w^t&*\\
0&\1_2&*\\
0&0&\e^{-\lambda_i}
\end{pmatrix} \in \O(1,3),
\]
with $w\in \rr^2$ and $\lambda_i$ one of the fixed real numbers.
Since these were chosen linearly independent over $\QQ$, the
quotient $\Hol_p(M)/\Hol(\tem)=\Hol_p(M)/\Hol^0(M)$ and also the holonomy group of the induced connection on the null line bundle $\cal L$ is infinitely
generated.

\subsection{Examples with curved screen bundle}
Using
constructions in Riemannian geometry one obtains Lorentzian
manifolds with parallel null line, curved screen bundle and
disconnected holonomy with a coupling between the $\rr^*$ and the
$\O(n)$ part.

The first construction is based on Riemannian manifolds that go back
to Hitchin \cite{hitchin74einstein} and McInnes
\cite{mcinnes91,mcinnes93} and which are described in detail in
\cite{moroianu-semmelmann00}. Let $N$ be the complete intersection
of $m+1$ quadrics in $\C P^{2p+1}$, defined as the common zero set
of quadratic polynomials with strictly positive real coefficients,
and endowed with the Riemannian metric $h$ from $\C P^{2p+1}$.
Then $(N,h)$ is a simply
connected K\"ahler manifold. On $N$ we have an involution defined by
complex conjugation,
\[\s ([z_0,\ldots ,z_{2p+1}]):= ([\ol{z}_0,\ldots ,\ol{z}_{2p+1}]),
\]
which is an anti-holomorphic isometry of $(N,h)$. Then, if $\Gamma$
denotes the group of isometries generated by $\s$, in
\cite[Corollary 9]{moroianu-semmelmann00} it is proven that
$(N/\Gamma,h)$ is a compact $2p$-dimensional Riemannian manifold
with holonomy $\Z_2\ltimes \SU(p)$.
In
addition, $(N/\Gamma,h)$ is Ricci flat and admits a parallel spinor
field. Note that $\Phi(\s)$ is represented by the
 conjugation on the tangent space with
respect to the complex structure given by the K\"ahler structure. We
denote this involution by $\s_*$.

If we fix a $\Gamma$-invariant function $f$ on $N$ and, for $m\in
\Z$ and $\lambda\in \rr$, perform both constructions  as in
Corollary~\ref{method-folge}, we obtain
$(\Z_2\ltimes\SU(p))\ltimes\rr^{2p}$ and
 $(\Z\ltimes\SU(p))\ltimes\rr^{2p}$
as holonomy groups, with $\Z_2$ and $\Z$ acting simultaneously on
$\rr$ and $\SU(p)$, i.e., \be
 \Hol(\tM/\Gamma_{m},g^{f,h})& = &\left\{ \left(%
\begin{array}{ccc}
 (-1)^{km}&  0& 0\\
  0 & \s_*^{k} & 0 \\
  0 & 0 & (-1)^{km}\\
\end{array}%
\right) \; \Big| \; k\in \Z, \right\}\cdot (\SU(p)\ltimes \rr^{2p}),
\\
 \Hol(\tM/\Gamma_{\lambda},g^{f/u^2,h})& = &\left\{ \left(%
\begin{array}{cccc}
 \e^{k\lambda}& 0& 0 \\
0 & \s_*^{k} & 0 \\
 0 & 0 & \e^{-k\lambda} \\
\end{array}%
\right) \; \Big| \; k\in \Z\right\}\cdot(\SU(p)\ltimes \rr^{2p}).
\ee Note that both manifolds no longer admit parallel spinors
because they do not admit a parallel  vector field.
We
will consider the full holonomy groups of Lorentzian manifolds with
parallel spinors in the next section.

In \cite{wilking99} Wilking constructed a remarkable example of a
compact Riemannian manifold with non-compact holonomy group. The
manifold is given as a quotient of a $5$-dimensional solvable Lie
group $S:=\rr^4\rtimes \rr$ with a left-invariant Riemannian metric
by a cocompact subgroup $\Gamma:=\Lambda\rtimes\Z $ where $\Lambda$
is a certain lattice of $\rr^4$. Then $S/\Gamma$ has restricted
holonomy $\SO(3)\subset \SO(5)$, acting trivially on a
$2$-dimensional subspace of $T_pS$, and $\SO(3)\times \Z$ as full
holonomy group.
 Again, by the constructions in Corollary~\ref{method-folge} we can produce $7$-Lorentzian manifolds with parallel null line
 and full holonomy  $(\Z\times \SO(3))\ltimes \rr^5$ where $\Z$ acts on $\rr$ and $\SO(3)$ simultaneously, either as $\Z_2$ or as $\Z$.

 In the same way as described in Remark~\ref{compact-remark} we can even construct compact
 Lorentzian manifolds of the form $(S^1\times S^1)/{\Z_2}\times S/\Gamma$ with indecomposable holonomy
 $(\Z_2\times \Z\times \SO(3))\ltimes \rr^5$
 or of the form $(S^1\times S^1\times S)/\Gamma_{\underline{m}}$, where $\Gamma_{\underline{m}}$
 is generated as in formula~\eqref{compact-coupled}, and with holonomy contained in $(\Z_2\times \Z\times \SO(3))\ltimes \rr^5$,
 but possible with non-trivial $\Z_2\subset \rr^*$-part which is  now coupled to $\Z\subset \O(5)$.

\section{Full holonomy groups of Lorentzian manifolds with parallel null spinor}
\label{section-hol-par}

A Lorentzian spin manifold $(M,g)$ of dimension $n+2>2$ is a time-
and space-oriented Lorentzian manifold with a spin structure
$\widetilde{\cal O} (M,g)\to \cal \O(M,g)$ that is a reduction with
respect to the double cover $ \Lambda:\Spin^0(1,n+1)\to
\SO^0(1,n+1)$, where $\cal \O(M,g)$ is the bundle of time- and
space-oriented frames over $M$ orthonormal for $g$. This allows us
to write the tangent bundle as
\[TM=\cal \O(M,g)\times_{\SO^0(1,n+1)}\rr^{1,n+1}=\widetilde{\cal O}(M,g)\times_{\Spin^0(1,n+1)}\rr^{1,n+1},\]
and  defines the spinor bundle
\[\Sigma=\widetilde{\cal O}(M,g)\times_{\Spin^0(1,n+1)}\Delta^{1,n+1},\]
where $\Delta^{1,n+1}$ is the spinor module. The spinor bundle is
equipped with a metric of neutral signature, a Clifford
multiplication
\[\cdot :TM\times \Sigma \to \Sigma,\]
and a covariant derivative $\nabla^\Sigma$. All these structures
are compatible with each other (for details see for example
\cite{baum81}).

Every section in the spinor bundle $\vf\in \Gamma(\Sigma)$ defines a
causal vector field $V_\vf\in \Gamma(TM)$ via
\[g(V_\vf,Y)= - \la Y\cdot \vf,\vf\ra.\]
Both, $\vf $ and $V_\vf$ have the same zero set, and the spinor
field  is of zero length, a {\em null spinor},  if  $V_\vf$  is of zero length, i.e. a null vector field. If the
spinor bundle admits a section $\vf$ with $\nabla^\Sigma \vf=0$, for
short, a parallel spinor, then $V_\vf$ is a parallel vector field as
well.

Since a Lorentzian spin manifold is assumed to be time- and
space-oriented, its full holonomy $H:=\Hol_p(M,g)$ is contained in
the connected component $\SO^0(1,n+1)$ of $\O(1,n+1)$, but is not
necessarily connected itself.  Using Proposition \ref{types} and results
in \cite{leistnerjdg}  and  \cite{wang89} we can summarise what we
know so far about the holonomy group in this situation:
 \bs
 Let $(M,g)$ be a Lorentzian  spin manifold with special holonomy.  If $(M,g)$ admits a parallel spinor, then
$(M,g)$ admits a parallel null vector field and its full holonomy
group $H$ is given as
\[H=G\ltimes \rrn,\]
with $G\subset \SO(n)$, and the connected component $G^0$ of $G$ is
trivial or given as
 a direct product of some of the following groups
  \[ \SU(m),\ \Sp(k), \ \G_2,\ \Spin(7).\]
  \es
  \bprf
The parallel spinor on  $(M,g)$ induces a parallel vector field,
which, as $H=\Hol(M,g)$ acts indecomposably, has to be null. Hence,
the full holonomy group fixes a null vector,
\[\Hol_p(M,g)\subset \mathrm{Stab}_{\SO^0(T_pM)}(V_\vf(p)) \simeq \SO(n)\ltimes \rrn,\]
and the restricted holonomy must be of type 2 or 4. Using the fact
that for the coupled type 4 the group $G^0$ has a centre, in
\cite{leistnerjdg} it was shown, that the existence of such a
parallel spinor implies that the restricted holonomy  must be of uncoupled type
2,
 \[H^0=\G^0\ltimes\rrn,\]
where $G^0$ is a Riemannian holonomy group which is isomorphic to a
subgroup in $\Spin(n)$ admitting a fixed spinor. Based on Berger's
list \cite{berger55}, these groups were determined  by Wang
\cite{wang89}.
  Using Proposition \ref{types} we obtain that $H=G\ltimes \rrn$ with $G\subset \SO(n)$.
  \eprf
We will now generalize this result to the full holonomy group.

\bs \label{satz-Ex-paral-spin} Let $(M,g)$ be a Lorentzian manifold
of dimension $(n+2)>2$ which is time- and space-orientable with
indecomposable restricted holonomy group $H^0\subset ( \rr^+\times
\SO(n)) \ltimes \rrn$. Then we have the following two implications:
\bnum
\item[1)] If $(M,g)$ admits a spin structure with a parallel spinor field, then the full holonomy group $H$ is given
as $H=G\ltimes \rrn$ with $G\subset\SO(n)$ and there exists a  homomorphism $\Phi:G\to \Spin(n)$ \bnum
\item[(a)]\label{a}  with $\lambda\circ \Phi=\mathrm{Id}_G$ for $\lambda:\Spin(n)\to\SO(n) $ the twofold cover, and
\item[(b)]\label{b} there is a  spinor $w$ in the spinor module $\Delta_n$ such that $\Phi(G)w=w$.
\enum
\item[2)] If the full holonomy group is $H=G\ltimes \rrn$ with $G\subset \SO(n)$ and there is a
 homomorphism $\Phi:G\to\Spin(n)$ with (a) and
(b), then $(M,g)$ has a spin structure with a parallel spinor.\enum \es

\bprf Let $(M,g)$ be time- and space-oriented and
$H=\Hol_p(M,g)\subset \SO^0(1,n+1)$ be its full holonomy group.
 The proof relies on the following three observations:
\bnum
\item[1)] \label{ob1}
If $(M,g)$ admits a spin structure with a parallel spinor field,
then there is a homomorphism $\Psi: H\to \Spin^0(1,n+1)$ with
$\Lambda\circ \Psi=\mathrm{Id}_H$ and a  spinor $v\in\Delta_{1,n+1}$
such that $\Psi(H)v=v$.
\item[2)]\label{ob2}
If there is a homomorphism  $\Psi:H\to \Spin^0(1,n+1)$ with
$\Lambda\circ \Psi=\mathrm{Id}_H$ and a spinor $v\in\Delta_{1,n+1}$
with $\Psi(H)v=v$,
  then $(M,g)$ has a spin structure with a parallel spinor on $(M,g)$.
  \item[3)]\label{ob3} If $H=G\ltimes \rrn$, then there is a homomorphism $\Psi: H\to \Spin^0(1,n+1)$
  with $\Lambda\circ \phi=\mathrm{Id}_H$
and a  spinor $v\in\Delta_{1,n+1}$ such that $\Psi(H)v=v$ if and
only if there is a homomorphims $\Phi:G\to \Spin(n)$ with (1a) and
(1b). \enum
The first observation was made by Wang in \cite{wang95}. Indeed, if
$(M,g)$ has a parallel spinor field, the holonomy group $\tilde{H}$
of the spin connection fixes a spinor $v\in \D_{1,n+1}$ and maps
onto $H$, hence it does not contain $-1\in Spin^0(1,n+1)$.
Therefore, $\Lambda|_{\tilde{H}}: \tilde{H} \to H$ is an isomorphism
and we can define $\Psi:=(\Lambda|_{\tilde{H}})^{-1}: H \to
Spin^0(1,n+1)$.

The second observation was made by Semmelmann and Moroianu in
\cite[Lemma 5]{moroianu-semmelmann00} for Riemannian signature but
their proof works in any signature. Indeed, if $\cal H$ is the
holonomy bundle through a frame in $\cal \O(M,g)$ and $\Psi:H\to
\Spin^0(1,n+1)$ is a homomorphism with $\Lambda\circ\Psi=\mathrm{Id}$, we can define the spin structure
by $\widetilde{\cal O}(M,g):=\cal H\times_H\Spin^0(1,n+1)$ which
projects canonically onto $\cal \O(M,g)=\cal
H\times_{H}\SO^0(1,n+1)$.

We have to proof the third observation. First assume that $\Psi:
H=G\ltimes \rrn\to \Spin^0(1,n+1)$ is given. Since
$\Lambda\circ\Psi=\mathrm{Id}_H$, the restriction of $\Psi$ to $G$
maps into $\Lambda^{-1}(\SO(n))=\Spin(n)\subset \Spin^0(1,n+1)$.
Hence we can define
 \[\Phi:=\Psi|_G:G\to \Spin(n).\]
Since  $\lambda=\Lambda |_{\Spin(n)}$, we also get that $\Phi\circ
\lambda=\mathrm{Id}_G$. Now, $\rrn\subset H$ is a connected closed
Abelian subgroup. When realising $\Spin^0(1,n+1)$ in the
 Clifford algebra $ {\cal Cl}(1,n+1)$,  the image of $\rrn$ under $\Psi$  is given by
\[\Psi (\rrn) = \{1+\ell\cdot x\mid x\in \rr^n\}\subset \Spin^0(1,n+1)\subset {\cal Cl}(1,n+1) ,\]
where $(\ell, e_1, \ldots , e_n,\ell^*)$ is a basis as in
\eqref{metric} and $\rr^n=\mathrm{span}(e_1, \ldots, e_n)$. On the
other hand, if $\Phi:G\to \SO(n)$ is given, we define \be \Psi :
H=G\ltimes \rrn&\to &\Spin^0(1,n+1)
\\g\cdot x&\mapsto& \Phi(g)\cdot (1+\ell\cdot x ).\ee
One can check that $\Psi\circ \Lambda=\mathrm{Id}_H$.

It remains to verify that  there is fixed spinor $v\in
\Delta_{1,n+1}$ if, and only if, there is  a fixed spinor $w\in
\Delta_n$. To this end we consider the Clifford algebra ${\cal
Cl}(1,1)$ of the 2-dimensional space $\R\ell \oplus\R\ell^*$ with
the induced signature $(1,1)$-scalar product, fix a basis
$(u_1,u_2)$ in $\Delta_{1,1}$ satisfying $\ell \cdot u_1 = \sqrt{2}
u_2$, $\ell \cdot u_2=0$, $\ell^*\cdot u_1= 0$, $\ell^* \cdot u_2=
-\sqrt{2}u_1$ and assign to a spinor $v\in \Delta_{1,n+1}$ two
spinors $v_1, v_2\in \Delta_n$ by identifying
\be\Delta_{1,n+1}&\simeq &\Delta_n\otimes \Delta_{1,1}\\
v&\mapsto&v_1\otimes u_1+v_2\otimes u_2. \ee Then a
computation shows that $(g\cdot a )(v)=v$ for all $g\in \Psi(G)$ and
$a\in \Psi(\rrn)$ if, and only if, \be gv_2 - \sqrt{2} ( g \cdot x )
 v_1&=&v_2,
\\
g v_1&=&v_1 \ee for all $g\in \Psi(G)$ and all $x\in \rrn$. The
first equation for $g=\1$ implies that $x\cdot v_1=0$ for
all $x\in \rrn$ and hence, $v_1=0$. Thus, $v\in \Delta_{1,n+1}$ is
fixed under $\Psi(H)$ if, and only if, $v=v_2\otimes u_2$ and $v_2$ is
fixed under $\Psi(G)=\Phi(G)$. This shows observation \eqref{ob3}
and completes the proof. \eprf

\bfolg \label{Folg-Number-parspinor} Let $H=G\ltimes \rrn\subset
\SO^0(1,n+1)$ with $G\subset SO(n)$. Then the following vector
spaces have the same dimension: \bnum
\item[i)] spinors in $\D_n$ fixed under $G$,
\item[ii)] spinors in $\D_{1,n+1}$ fixed under $H$,
\item[iii)]  parallel spinors fields on a Lorentzian manifold with holonomy group $H=G\ltimes\rrn$.
\enum \efolg

Finally, what is needed to complete a proof  of Theorem \ref{theo2}
is a classification of subgroups of $\SO(n)$ that fix a spinor in
$\D_n$ and have connected component $\SU(m)$, $\Sp(k)$, $\G_2$ and
$\Spin(7)$. This result can be obtained from \cite{mcinnes91} and
\cite{wang95}.

\btheo[McInnes \cite{mcinnes91} and
Wang \cite{wang95}]\label{wangtheo} Let $G\subset \SO(n)$ be Lie
group with connected component $G^0$ equal to $\SU(m)$, $\Sp(k)$,
$\G_2$ or $\Spin(7)$ with a non-vanishing fixed spinor in
$\Delta_n$. Then $G$ is equal to one of the groups in the following
table, in which $N$ is the dimension of spinors fixed under $G$,
\vspace{2mm}
\begin{center}
\begin{tabular}{|l|l|l|l|l|}
\hline
$G^0$    &  $n$   &  $G$  &  $N$   &  \text{conditions} \\[1mm]
 \hline
 &&&&\\[-2mm]
$\SU (m)$            &    $2m$    &    $\SU(m)$    &  2     &  \\[1mm]
\cline{3-5}
 &&&&\\[-2mm]
                  &  & $\SU(m) \rtimes \Z_2$ & 1    & $m$ divisible by $4$
 \\[1mm]
 \hline
  &&&&\\[-2mm]
           &      &     $\Sp(k)$    &  $k+1$ &
           \\[1mm]
            \cline{3-5}
             &&&&\\[-2mm]
$\Sp(k)$  & $4k$  &   $\Sp(k) \times \Z_d$   & $(k+1)/d$ & $d > 1, d $
odd and divides $  k+1$
 \\[1mm]
  \cline{3-5}
   &&&&\\[-2mm]
&  & $\Sp(k) \cdot \Z_{2d}$ & $2\left\lfloor \tfrac{k}{2d}\right\rfloor +1 $& $k$ even, $1<d\le 2d$
\\[1mm]
\cline{3-5}
 &&&&\\[-2mm]
&  & $\Sp(k) \cdot Q_{4d}$ &$ \left\lfloor \tfrac{k}{2d}\right\rfloor$ if $\tfrac{k}{2}$ odd& $k$ even, $1<d\le 2d$,
\\[1mm]
\cline{4-4}
&&&$ \left\lfloor \tfrac{k}{2d}\right\rfloor +1$ if $\tfrac{k}{2}$ even&\\[1mm]
\cline{3-5}
 &&&&\\[-2mm]
&  & $\Sp(k) \cdot B_{4d}$ & see ref. \cite{wang95} & $k$ even and conditions in \cite{wang95}
\\[1mm]
\cline{3-5}
 &&&&\\[-2mm]
&  & $\Sp(k) \cdot \Gamma$ &1 & $k$ even
\\[1mm]
\hline
 &&&&\\[-2mm]
$\Spin_7$          &   8  &    $\Spin_7$        &  1     &
\\[1mm]
\hline
 &&&&\\[-2mm]
$\mathrm{G}_2$             &    7    &    $\mathrm{G}_2$        &  1     &             \\
\hline
\end{tabular}
\end{center}
\vspace{2mm}
Here
\bnum
\item $Q_{4d}$ is the double cover of the dihedral group $D_{2d}$ of order $2d$,
\item
$\Sp(k)\cdot B_{4d}$ for $d=6,12,30$, and $B_{4d}$ is the double
cover in $\Sp(1)$ of the polyhedral groups $P_{2d}$ in $\SO(3)$,
i.e. the tetrahedral group $P_{12}$, the octahedral group $P_{24}$,
and the icosahedral group $P_{60}$, and
\item
 $\Gamma$ is an infinite subgroup of $\U(1)\rtimes \Z_2$.
 \enum
\etheo
\bprf[Steps in the proof] Since $G$ is contained in the normaliser
of $G^0$ in $\O(n)$, first we need a list of normalisers of the
possible $G^0$'s. They can be found in \cite[10.114]{besse87} with a
correction made in \cite{mcinnes91} for the $\SU(m)$-case. The cases
in which $G^0$ is equal to $\G_2$ or $\Spin(7)$ are trivial, as both
groups are equal to their own normaliser in $\O(n)$. We are left
with $G^0$ being $\SU(m)$ or $Sp(k)$. Their normaliser in $\O(n)$ is
given as
$\U(m)\rtimes \Z_2$, where $\Z_2$ acts by complex conjugations, and
as
$
\Sp(k)\cdot Sp(1)$.

First assume that $G/G^0$ is finite.  In  \cite{mcinnes91}, McInnes
classified  the possible holonomy groups of compact Ricci flat
Riemannian  manifolds. Since their fundamental group is finite, as
the first, {\em purely algebraic} step in McInnes' proof, possible
subgroups $G$ in $\SU(m)\rtimes \Z_2$ and $\Sp(k)\cdot Sp(1)$ with
finite quotients $G/\SU(m)$ and $G/\Sp(k)$ are listed. For  $\SU(m)$
they are of the form
\[ \Z_{mr}\cdot \SU(m)\ \text{ or }\ \left(\Z_{mr}\cdot \SU(m)\right)\rtimes \Z_2,\]
with a positive integer $r$ and with $\Z_{mr}\in \U(1)$. For
$\Sp(k)$ the list is longer: \bnum
\item[(i)]
$\Z_r\cdot \Sp(k)$, with $r$ odd,
\item[(ii)]
$\Z_{2r}\cdot \Sp(k)$, with $r$ even,
\item[(iii)] $Q_{4d}\cdot \Sp(k)$, where $Q_{4d}$ is the double cover of the dihedral group $D_{2d}$ of order $2d$,
\item[(iv)] $B_{4d}\cdot \Sp(k)$ for $d=6,12,30$, and $B_{4d}$ is the double cover in $\Sp(1)$ of the polyhedral
groups $P_{2d}$ in $\SO(3)$, i.e. the tetrahedral group $P_{12}$,
the octahedral group $P_{24}$, and the icosahedral group $P_{60}$.
\enum Using geometric arguments McInnes shortened this list to
obtain all possible holonomy groups of {\em compact} Ricci flat
Riemannian manifolds, but since we cannot apply geometric arguments
for our purpose, we cannot use this shorter list. Instead we use
results by Wang in \cite{wang95}, where the full holonomy groups of
Riemannian manifolds --- compact and non-compact --- with parallel
spinors are classified. Although in the compact case Wang can start
from the shorter list obtained by McInnes, for the non-compact case
in \cite[Proof of Theorem 4.1]{wang95} only algebraic arguments can
be used and the full list above has to be checked for the existence
of fixed spinors. In the $\SU(m)$ case Wang shows that $r=0$, that
is, only $\SU(m)$ itself and $\SU(m)\rtimes \Z_2$ remains. In the
$\Sp(m)$ case Wang obtains the list in the table in the Theorem. \eprf

To conclude this section we consider the question whether there
exist Lorentzian manifolds with the holonomy groups in Theorem
\ref{theo2} and special causality properties. In Proposition
\ref{satz-realisation-standard} we proved that starting with a
Riemannian manifold $(N,g_N)$ with full holonomy group $G$ and a
function  $f\in \C^\infty(\R\times N)$ such that $\mathrm{det}(
\mathrm{Hess^N}(f))|_{p} \not=0$ at some point $p\in \R\times N$, we
obtain a Lorentzian manifold $M:= \rr^2 \times N$ with the metric
\begin{eqnarray} \label{general-pp-spinor}
g^{f,h} =2 \d v\d u + 2f \d u^2+ g_N \end{eqnarray}
 with full holonomy
$G\ltimes \rr^n$. Proposition \ref{satz-Ex-paral-spin} and Corollary
\ref{Folg-Number-parspinor} show, that in case of a spin manifold
$(N,g_N)$, the Lorentzian manifold $(M,g^{f,h})$ is spin as well and the
dimension of the spaces of parallel spinor fields on $(M,g^{f,h})$ and
$(N,g_N)$ are the same.  Moreover, for Lorentzian manifolds of type
$(M,g^{f,h})$ various causality properties are known (see for example
\cite{Sanchez1} and \cite{Sanchez2}). Let us quote here the
following two results.
\begin{enumerate}
\item[1)]
If $(N,g_N)$ is a complete Riemannian manifold, the function $f$
does not depend on $u$ and is at most quadratic at spacial infinity,
i.e., there exist $x_0\in N$ and real constants $r,c >0$ such that
\[ f(x) \leq c \cdot d_N(x_0,x)^2 \quad \mbox{ for all } x\in N \mbox{ with } d_N(x_0,x) \geq r,\]
then $(M,g^{f,h})$ is geodesically complete. Here $d_N$ is the
distance function of $(N,g_N)$.
\item[2)] If $(N,g_N)$ is a complete Riemannian manifold and the function
$-f$ is spacial subquadratic, i.e., there exist $x_0\in N$ and
continious functions $p, c_1, c_2 \in C(\R,[0,\infty))$ with $p(u)<
2$ such that
\[ -f(u,x) \leq c_1(u)\, d_N(x_0,x)^{p(u)} + c_2(u) \quad \mbox{ for all} \; (u,x) \in \R\times N, \]
then $(M,g^{f,h})$ is globally hyperbolic.
\end{enumerate}
Of course, both conditions for $f$ can be realized in addition to
$\mathrm{det}\big( \mathrm{Hess}^N f(u_0,x_0)\big) \not = 0$. Hence,
each of the groups in Theorem \ref{theo2} can be realized as
holonomy group of a Lorentzian manifold, and in addition, if the
group $G$ in Theorem \ref{wangtheo} is the holonomy group of a {\em
complete} Riemannian manifold, then $H$ can be realized by a
geodesically complete as well as by a globally hyperbolic Lorentzian
manifold.

If one is interested in globally hyperbolic manifolds with complete
or even compact space-like Cauchy surfaces, another construction
based on {\em Lorentzian cylinders} is useful, which from the spin
geometric point of view first was studied by B\"ar, Gauduchon and
Moroianu in \cite{baer-gauduchon-moroianu05} and further developed
in the context of special holonomy by the first author and M\"uller
in \cite{baum-mueller08}. Formulated for our situation the result
is:

\begin{satz}[\cite{baum-mueller08}] \label{Satz-Realisation-global-hyp}
Let $(N,g_N)$ be an n-dimensional irreducible Riemannian spin
manifold of dimension $n$ with parallel spinors, $(F,g_F)$ the
warped product $( F = \R \times N, g_F=ds^2 + e^{-4s}g_N )$ over
$(N,g_N)$, $\,C:TF \to TF$ a Codazzi tensor on $(F,g_F)$ with only
positive eigenvalues and $a\in \R$ a positive constant. Then the
Lorentzian manifold $(M,g^C)$ given by
\begin{eqnarray} M:=
(-a,\infty) \times \R \times N, \qquad g^C:= -dt^2 + \big(C +
2(t+a)\mathrm{Id}_{TF}\big)^* g_F \label{metric-glob-hyp}
\end{eqnarray}
has full holonomy
\[ \Hol_{(0,0,p)}(M,g^C) = C^{-1} \circ \Hol_p(N,g_N) \circ C) \ltimes
\R^n. \] Moreover, if $(N,g_N)$ is complete, then the Lorentzian
manifold $(M,g^C)$ is globally hyperbolic and the space-like slices
$\,( \{t\} \times F, \,g^C_t = (C+2(t+a)\mathrm{Id}_{TF})^*g_F )$
are complete Cauchy surfaces.
\end{satz}
\begin{proof} The proof in \cite{baum-mueller08}, Theorem 3, states the
result for the reduced holonomy groups. Using Proposition
\ref{Satz-full-hol} in addition, we obtain the result for the full
holonomy group. \end{proof}
Explicit examples for Codazzi tensors $C$ on the warped product
$(\R\times N, ds^2 + e^{-4s}g_N)$ are given in
\cite{baum-mueller08}. Take for example a bounded, strictly
increasing function $f\in C^{\infty}(\R)$ with $f(0)=0$ and
$f(s)<\lambda$ for all $s\in \R$. Then $C^f: \R\del_s \oplus TN \to
\R\del_s \oplus TN$ given by
\[ C^f := \left(%
\begin{array}{cc}
  e^{2s}f'(s) & 0 \\
  0 & 2e^{2s}(\lambda - f(s)) \mathrm{Id}_{TN} \\
\end{array}%
\right) \] is a Codazzi tensor on $(F,g_F)$ and the metric
(\ref{metric-glob-hyp}) is given by
\[ g^{C^f} = -dt^2 + \big(e^{2s} f'(s) + 2a +2t \big)^2 ds^2 + 4 \big(
e^{-2s}t + e^{-2s}a + \lambda -f(s) \big)^2 g_N .\]

These two constructions reduce the problem of finding, for each $G$
in the table in Theorem~\ref{theo1}, a Lorentzian manifold with
holonomy $G\ltimes \R^n$ to the Riemannian case. First, one has to ensure
the existence of Riemannian manifolds with holonomy group $G$.
Then, for geodesically complete or globally hyperbolic Lorentzian
metrics, one needs complete Riemannian manifolds with holonomy group
$G$. For connected holonomy groups we can built on the deep
existence results for complete and even compact Riemannian manifolds
with special holonomy obtained by several authors (for an overview
see  \cite{joyce07}). Based on the examples with
connected holonomy groups, Moroianu and Semmelmann in
\cite{moroianu-semmelmann00} constructed Riemannian
manifolds with parallel spinor for each of the non-connected groups
$G$ in the table in Theorem \ref{theo2}. For $\SU(m)\rtimes \Z_2$
they construct a compact manifold, and for the remaining groups the
metrics are obtained by removing points from compact spaces or by
cone constructions, thus these metrics are not complete. This yields
the following conclusion.

\bfolg For each of the groups $G$ in Theorem \ref{theo2} there exist
Lorentzian manifolds with holonomy $G\ltimes \rrn$ and parallel
spinors. Moreover, for the connected groups $G$ and for
$\SU(m)\rtimes \Z_2$, there exist geodesically complete as well as
globally hyperbolic Lorentzian manifolds with complete spacelike
Cauchy surfaces and holonomy $G\ltimes \rrn$. \efolg
It would be interesting to know, if the groups $Sp(m)\times \Z_d$,
and $Sp(m) \cdot \Gamma$ in Theorem \ref{wangtheo} can be realized as
holonomy group of a {\em complete} Riemannian manifold.

%
%
%
%
%
%

%

%
%

%


\def\cprime{$'$}

\end{document}